\documentclass{amsart}   
\usepackage{amssymb,amsmath,amsfonts,amscd}

\usepackage{graphicx}
\usepackage{color}

\definecolor{darkgreen}{cmyk}{1,0,1,.2}
\definecolor{m}{rgb}{1,0.1,1}

\def\cE{{\mathcal E}}
\def\xg{{X^G_{\mathcal{P}(\omega)}}}
\def\hu{${X^G_{T}}$}
\def\cT{{\mathcal T}}
\def\cP{{\mathcal P}}
\def\H{{\mathbb H_{2}}}
\def\rt{{\rtimes}}
\def\af{\mathcal{A}}

\def\afx{\mathcal{A}^G_{\mathcal{P}(\omega)}}
\def\al{{\alpha}}
\def\be{{\beta}}

\def\rta{{\rtimes_\alpha}}

\newcommand\co{\operatorname{Coinv}}
\newcommand\inv{\operatorname{Inv}}
\newcommand\sh{{\#}}
\newcommand\xb{{X_{\beta^0}}}
\newcommand\yb{{Y_{\beta^0}}}
\newcommand\zb{{Z_{\beta^0}}}
\newcommand\zbu{{Z_{\beta^1}}}

\newcommand\C{\mathbb C}
\newcommand\cN{\mathcal N}
\newcommand\cA{\mathcal A}
\newcommand\N{\mathbb N}
\renewcommand\P{\mathcal{P}}
\newcommand\R{\mathbb R}
\newcommand\cR{\mathcal R}
\newcommand\Z{\mathbb Z}
\newcommand\E{\mathcal E}
\newcommand\T{\mathcal{T}}
\newcommand\K{\mathcal{K}}
\newcommand\ts{{\otimes}}
\newcommand\si{\sigma}

\newcommand\Ga{{\Gamma}}

\newcommand\ga{{\gamma}}
\newcommand\lto{{\longrightarrow}}
\newcommand\defi{{\stackrel{\text{def}}{=\!=}}}

\newcommand\G{\mathcal{G}}

\renewcommand\L{\mathcal{L}}
\theoremstyle{plain}
\newtheorem{theorem}{Theorem}[section]
 
\newtheorem{proposition}[theorem]{Proposition}
\newtheorem{corollary}[theorem]{Corollary}

\newtheorem{lemma}[theorem]{Lemma}
\newtheorem{remark}[theorem]{Remark}
\newtheorem{definition}[theorem]{Definition}
\newtheorem{defn}[theorem]{Definition}
\newtheorem{prop}[theorem]{Proposition}
\newtheorem{ex}[theorem]{Example}

 \title[$C^*$-algebras of Penrose's hyperbolic tilings]{$C^*$-algebras
   of Penrose hyperbolic tilings}  
 \author[Oyono-Oyono]{Herve Oyono-Oyono}
\address{Laboratoire de Math\'ematiques, Universit\'e Blaise Pascal \&
  CNRS (UMR 6620),
Les C\'ezeaux, 24 avenue des Landais,
BP 80026, 63177 Aubi\`ere Cedex, France, phone:+33.4.73.40.70.71, fax:+33.4.73.40.70.64}
 \email{oyono@math.cnrs.fr}

 \author[Petite]{Samuel Petite}
\address{L.A.M.F.A, Universit\'e de Picardie Jules Verne \& CNRS, Amiens,
  France,
33 rue Saint-Leu,
80039 Amiens Cedex 1,
France,
phone:+33.3.22.82.78.15
fax:+33.3.22.82.78.38}
 \email{samuel.petite@u-picardie.fr}
\numberwithin{equation}{section}
 
\begin{document}
\begin{abstract} Penrose hyperbolic tilings are tilings of the
  hyperbolic plane which admit, up to affine transformations a finite
  number of prototiles. In this paper, we give  a complete
  description of the $C^*$-algebras and of the $K$-theory for  such tilings. Since the continuous hull of these tilings have no
  transversally invariant measure, these  $C^*$-algebras are traceless. Nevertheless, harmonic currents give
  rise to $3$-cyclic cocycles and we discuss in this setting a
  higher-order version of the gap-labelling.
\end{abstract}\maketitle
\begin{flushleft}{\it {\bf Keywords:}  Hyperbolic aperiodic tilings,
    $C^*$-algebras of dynamical systems, K-theory, Cyclic cohomology}


{\it {\bf Subject Classification:} Noncommutative topology
    and geometry, Dynamical systems} 
\end{flushleft}

\section{Introduction}
The non-commutative geometry of a quasi-periodic tiling studies an
appropriate $C^*$-algebra of a dynamical system $(X,G)$, for a compact
metric space $X$, called {\em the hull}, endowed with a continuous Lie
group $G$  action. This $C^*$-algebra is of relevance to study the
space of leaves which is pathological in any topological sense.
The hull owns also a geometrical structure of  {\em lamination} or
{\em foliated space}, the transverse structure being just metric \cite{Gh}.  The $C^*$-algebras and the non-commutative tools provide then topological and geometrical invariants for the tiling or the lamination. Moreover, some $K$-theoretical invariants of Euclidean  tilings have a physical interpretation. In particular, when the tiling represents a  quasi-crystal,  the image of the  $K$-theory 
 under the canonical trace 
  labels  the gaps in the spectrum of the Schr\"odinger operator associated with the quasi-crystal \cite{Bel}. 

For an Euclidean tiling,  the group $G$ is $\R^d$ and $\R^d$-invariant
ergodic probability  measures on the hull are in one-to-one
correspondence with  ergodic  transversal invariant measures  and also
with extremal traces on the $C^*$-algebra \cite{BBG}. These algebras are well  studied and this leads, for instance,  to give distinct proofs of  the {\em gap labelling conjecture} \cite{BBG, BeO, KaPu}, i.e.  for minimal $\R^d$-action, the image of the $K$-theory  
under a trace  is the countable subgroup of $\R$ generated by the
images under the corresponding transversal invariant measure  of  the compact-open
subsets   of the (Cantor)  canonical transversal.

For a hyperbolic quasi-periodic tiling, the situation is quite
distinct. The group of affine transformations acts on the hull  and since this
group is not unimodular, there is no transversally invariant measure
\cite{Pl}. A new phenomena    shows up  for  the $C^*$-algebra of the
tiling:  it has no trace. Nevertheless, the affine group is amenable,
so the hull admits at least one  invariant probability measure. These
measures are actually in one-to-one correspondence with   harmonic
currents  \cite{Pe}, and they  provide $3$-cyclic cocycles on the
smooth algebra of the tiling.

The present paper is devoted to give a complete description of the
$C^*$-algebra and the $K$-theory of a specific family of hyperbolic
tilings derivated from the example given by Penrose in \cite{Pen}. The
dynamic of the hulls under investigation, have a structure of double
suspension (this make sens in term of groupoids as we shall see in
section \ref{subsec-groupoid-suspension}) which enables to make
explicit  computations.  {This suggests  that the pairing with the
$3$-cyclic cocycle is closely related to the}
 one-dimension gap-labelling
for { a subshift}  associated with the tiling. But the  right setting to
state an analogue of the gap-labelling  seems to be    Frechet
algebras and a natural question is   whether this bring in new computable invariants.

Background on  tiling spaces is given in the next section and  we construct examples of hyperbolic quasi-periodic tilings in the third section.  A description of the considered hulls is given in
section 4. In section 5, we recall the background on the groupoids
and their $C^*$-algebras. {Sections 6 and 7 are devoted to the complete
description of the $C^*$-algebras of the examples and   their $K$-theory groups in terms of generetors are given in section $8$. For readers interested in topological invariants of the hull, we compute  its  $K$-theory and its  C\v ech cohomology   and we relate these computations to the former one.} In the last  section   we construct  $3$-cyclic cocycles associated to these tilings and  we discuss an odd version of the gap-labelling.

\section{Background on tilings}

Let $\H$ be the real  hyperbolic $2$-space,  identified with the upper half
complex plane: $\{ (x,y) \in \R^2 \ | y> 0 \}$ with the metric $ds^2 =
\frac{dx^2+dy^2}{y^2}$.  We denote by $G$ the group of {\em affine transformations} of this space: i.e.   the isometries of $\H$ of the  kind $z \mapsto az+b$ with $a$, $b$ reals and $a>0$.

A tiling $T= \{t_1, \ldots, t_n, \ldots \}$ of $\H$, is a collection of convex compact polygons
$t_i$ with geodesic borders, called {\em tiles}, such that their union is the whole space $\H$,
their interiors are pairwise disjoint and they meet full edge to full edge.  {For instance, when $F$ is a fundamental domain of a  co-compact lattice $\Gamma$ of isometries of $\H$, then $\{\gamma
(F), \ \gamma \in \Gamma \}$ is a tiling of $\H$.  However the set of tilings is much richer
than the one given by lattices as we should see later on. Similarly to the
Euclidian case, a  tiling is said of $G$-{\em finite type} or {\em finite affine type},  if there exists a finite number of polygons $\{ p_1, \ldots, p_n \}$ called {\it prototiles} such
that each $t_i$ is the image of one of these polygons by an element of $G$.} 
  Besides   its famous Euclidean tiling, Penrose in
\cite{Pen}  constructs a finite affine type tiling made with a single prototile which is not stable for
any Fuchsian group. The construction goes as follows.

\subsection{Hyperbolic Penrose's tiling}\label{hPT}

\noindent Let $P$ be the convex polygon with vertices $A_p$ with affix $(p-1)/2+ \imath$ for $1\leq p
\leq 3$ and $A_{4} :2 \imath+ 1$ and $A_5 :2\imath $  $P$ is a polygon with $5$
geodesic edges. Consider the two maps:
$$ R: z \mapsto 2z \ {\rm and }\ S: z \mapsto z+1. $$

  \begin{figure}[t]
  \begin{center}
  \includegraphics{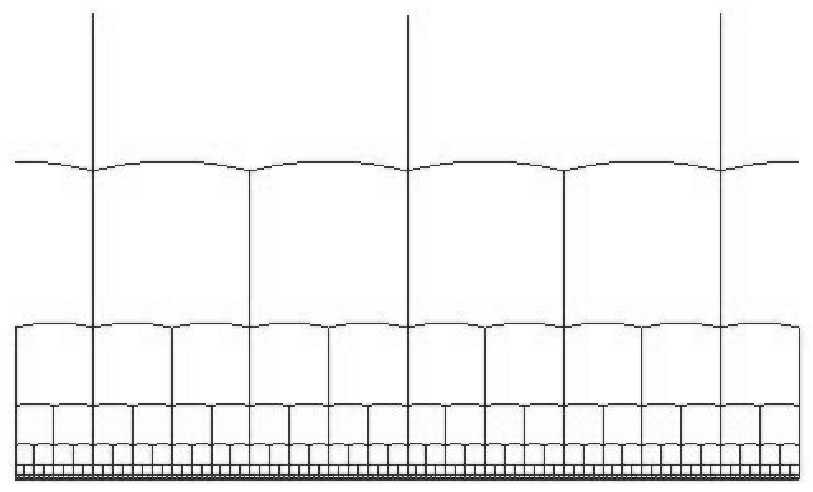}
  \end{center}
  \caption{The hyperbolic Penrose's tiling} \label{Pentil}
  \end{figure}

\noindent The hyperbolic Penrose's tiling is defined by ${\mathcal P}=\{ R^{k}\circ S^n P | \ n,k \in
\Z \}$ (see figure \ref{Pentil}). This is an example of finite affine type  tiling of $\H$. 

This tiling is stable under no  co-compact group of hyperbolic isometries. The proof
is homological: we associate with the edge $A_4A_5$ a
positive charge and two negative charges with edges $A_1A_2$, $A_2A_3$. If $\cP$ was stable for a
Fuchsian group, then $P$ would tile a compact surface. Since the edge $A_4A_5$ can meet only the
edges $A_1A_2$ or $A_2A_3$, the surface has a neutral charge. This is in contradiction with the
fact $P$ is negatively charged.

\noindent G. Margulis and S. Mozes \cite{MarMoz} have generalized this construction to build a
family of prototiles which cannot be used to tile a compact surface. Notice the group of isometries
which preserves $\P$ is not trivial and is generated by the transformation $R$. In order to break this symmetry, it
is possible, by a standard way, to decorate prototiles to get a new
finite affine  type tiling which is stable under no
non-trivial isometry (we say in this case that the tiling is {\em aperiodic}).


\section{Background on tiling spaces}

In this section, we recall some basic definitions and properties on
dynamical systems  associated with tilings. We refer to  \cite{G},
\cite{KP} and \cite{Ro}   for the proofs. We give then  a description of the dynamical system associated to the hyperbolic Penrose's tiling. 

\subsection{Action on tilings space}

\noindent First, note that the group $G$  acts transitively, freely (without a fixed point) and preserving the
orientation of the surface $\H$, thus $G$ is a Lie group homeomorphic to $\H$. The metric on $\H$
gives a left multiplicative invariant metric on $G$. We fix the point $O$ in $\H$ with affix $i$ that we call {\it origin}.

\noindent For a tiling $T$ of $G$ finite type and an isometry $p$ in $G$, the image of $T$ by
$p$ is again a tiling of $\H$ of $G$ finite type. We denote by $G.T$ the set of tilings which
are image of $T$ by isometries in $G$. The group $G$ acts on this set by the left action:

$$\begin{array}{ccc}
  G \times G.T & \longrightarrow & G.T \\
  (p, T') & \longmapsto & p.T' = p(T').
\end{array}
$$

\medskip

We equip $G.T$ with a metrizable topology, {so that the action becomes continuous}.  A base of
neighborhoods is defined as follows: two tilings are close one to  the other if they agree, on a big
ball of $\H$ centered at the origin,  up to an isometry in $G$ close to the identity. This topology
can be generated by the metric $\delta$ on $G.T$ defined by (see \cite{G}):

\noindent For $T$ and $T'$ be two tilings of $G.T$,  let
 $$A = \{ \epsilon \in (0, \frac{1}{\sqrt{2}}] \  \vert\  \exists\  g \in B_{\epsilon}(Id) \subset G \ {\rm s.t.} \  g.T\cap
B_{1/\epsilon} = T' \cap B_{1/\epsilon}\} $$

\noindent where $B_{1/\epsilon}$ is the set of points $x\in
\H{\cong G}$ such that $d(x, O) <
1/{\epsilon}$.

 \noindent   We define: $$ \delta(T,T') = {\rm{\hbox{inf }}}
A {\rm \hbox{     if  }} A \not= \emptyset$$
   $$\delta(T, T') = \frac{1}{\sqrt{2}} {\rm \hbox{     }else}.$$

\medskip
\noindent The {\it continuous hull} of the tiling $T$, is the metric completion of $G.T$ for the
metric $\delta$. We denote it by $X^G_T$. Actually this space is a set of tilings of $\H$ of
$G$-finite type. A {\it patch} of a tiling $T$ is a finite set of tiles of $T$. It is
straightforward to show that patches of tilings in \hu \  are copies of  patches of $T$. The set
\hu \ is then a compact metric set and the action of $G$ {on $G\cdot T$}
can be extended to a continuous left
action on this space. The dynamical system $(X^G_T, G)$ has a dense orbit:  the orbit of $T$.\\
Some combinatorial  properties can be interpreted in a dynamical way like, for instance, the following. 


\begin{defn}
A tiling $T$ satisfies the repetitivity condition if for each patch $P$, there exists a real $R(P)$
such that every ball of $\H$ with radius $R(P)$ intersected with the
tiling $T$ contains a translated by an element $G$ of
the patch $P$.
\end{defn}

\noindent This definition can be interpreted from a dynamical point of view (for a proof see for instance
\cite{GH}).

\begin{prop}[Gottschalk]
The dynamical system $(X^G_T,G)$ is minimal (any orbit is dense) if and only if the tiling
$T$ satisfies the repetitivity condition.
\end{prop}

We call a tiling {\it aperiodic} if the action of $G$ on \hu \ is free: for all $p \neq Id$ of $G$ and
all tilings  $T'$ of \hu \  we have $p.T' \neq T' $. 

\noindent  As we have seen in the former section  the hyperbolic
Penrose's tiling is not aperiodic, however, using this example, we shall construct in section \ref{exemple} uncountably many examples of repetitive and aperiodic affine finite type tilings.

When the tiling $T$ is aperiodic and repetitive, the hull \hu \ has also a geometric structure  of a specific lamination called a $G$-{\em solenoid} (see \cite{G}).  Locally at any point $x$, there exists a {\em vertical germ} which is a Cantor set included in \hu, transverse to the local  $G$-action and which is  defined independently  of the neighborhood of the point $x$. This implies that \hu \ is locally  homeomorphic to the  Cartesian product of a Cantor set with an open subset  (called a {\em slice}) of the Lie group $G$.  The connected component of the slices that intersect is called a {\em leaf} and has a manifold structure. Globally, \hu \ is a disjoint union of uncountably many leaves, and it turns out that  each leaf is a $G$-orbit. Since the action is free,   each leaf is homeomorphic to $\H$.

In the aperiodic case, the $G$-action is {\it expansive}: There exists a positive real $\epsilon$ such that for every
points $T_1$ and $T_2$ in the same vertical in \hu, if $\delta(T_1.g,T_2.g) < \epsilon$ for every
$g \in G$, then $T_1=T_2$.\\
Furthermore this action has locally constant return times: if an orbit (or a leaf) intersects two
verticals $V$ and $V'$ at points $v$ and $v.g$ where $g \in G$, then for any point $w$ of $V$
close enough to $v$, $w.g$ belongs to $V'$.

\subsection{Structure of the hull of the Penrose Hyperbolic tilings}\label{subsec-hull}

First recall the notion of suspension action for $X$  a  compact metric space and $f : X \to X$ a homeomorphim.  The group $\Z$ acts diagonally on the product space $X \times \R$  by the following homeomorphism denoted $\cA_{f}$
\begin{eqnarray*}
\cA_{f } : \  X \times \R \to X \times\R \\
 (x, t) \mapsto (f(x), t-1)
\end{eqnarray*}
The quotient space of $(X \times \R) \slash \cA_{f}$, where two points are identified if they belong to the same orbit, is a compact set  for  the product topology and is  called {\em the suspension} of $(X,f)$. The group $\R$ acts also diagonally on $X \times \R$: trivially on X and by translation on $\R$. Since this action commutes with  $\cA_{f}$, this induces a continuous $\R$ action on the suspension space    $(X \times \R) \slash \cA_{f}$ that we call:  the {\em suspension action } of the system $(X,f)$ and we denote it by $\big( (X \times \R) \slash\cA_{f}, \R \big)$.

We recall here, the construction of the  dyadic completion of the integers. On the set of integers $\Z$, we consider the dyadic norm defined by $$ \vert n \vert_{2}= 2^{-\sup\{ p \in \N, \ 2^p\  \textrm{divides} \  \vert n\vert \}} \hspace{1cm} n\in \Z.$$ 
Let $\Omega$ be the completion of the set $\Z$ for the metric given by $\vert .\vert_{2}$. The  set $\Omega$ has a commutative group structure where $\Z$ is a dense subgroup, and $ \Omega$ is a Cantor set. The continuous action given by the map $o: \  x \mapsto x+1$ on $\Omega$ is called {\em adding-machine} or {\em  odometer} and is known to be minimal and equicontinuous. We denote by   $((\Omega \times \R )\slash \cA_{o}, \R)$ the suspension action of this homeomorphism.\\

Recall that a conjugacy map between two dynamical systems is a homeomophism which commutes with the actions. Let $\cN$ be the group of {transformations} $\{ z \mapsto z+t, \ t\in \R\}$ isomorphic to $\R$.

\begin{prop}\label{premsuspension} Let $X^\cN_{\P}$ be the closure
  (for the tiling topology) of the orbit $\cN. \P\subset  X^G_{\P}$. Then the   dynamical system $(X^\cN_{\P}, \cN)$  is conjugate to the suspension action of the odometer $((\Omega \times \R) \slash\cA_{o}, \R)$.
\end{prop}
\begin{proof} Let $\phi : \cN. \P \to (\Omega \times \R)/\cA_{o}$   be
the map defined by $\phi (\P + t) = [ 0 , t]$ where  $[0 , t]$ is the
$\cA_{o}$-class 
of $(0, t) \in\Omega \times \R$. Since the tiling  is  invariant under
no   translations, the  application $\phi$ is well defined. It is
straightforward to check that $\phi$ is continuous for the  tiling
topology and for {the topology on $(\Omega \times \R)/\cA_{o}$  arising
from} the dyadic topology on $\Omega$.  So the map $\phi$
extends by continuity to $X^\cN_{\P}$. Let us check that $\phi$ is a
homeomorphism by constructing its inverse.
Let $\psi : \Z \times \R \to X^\cN_{\P}$ defined by $\psi (n, t) =
\P+n+t$. This application  is continuous for the dyadic topology,
 so it extends by continuity to $\Omega \times \R$. Notice
 $\psi(n,t)$ 
is constant along the orbits of the $\cA_{o}$ action on $\Z\times \R$
which is dense in $\Omega \times \R$.
 Thus $\psi$ is constant along the $\cA_{o}$-orbits in $\Omega \times
 \R$  and $\psi$ factorizes through  a map 
$\overline{\psi}$ from the suspension $(\Omega\times\R)/\cA_{o}$ to
$X^\cN_{\P}$. 
It is plain to check that $\overline{\psi}\circ{\phi} = Id$ on the
dense set $\cN. \P$ and that  
${\phi}\circ\overline{\psi} = Id$ on  the dense set $\pi(\Z \times
\R)$ where 
$\pi : \Omega\times \R \to  (\Omega\times \R) / \cA_{o}$  denotes the
canonical projection. 
Hence ${\phi}$ is an homeomorphism from $X^\cN_{\P}$ onto
$(\Omega \times \R)/ \cA_{o}$. 
It is obvious that  ${\phi}$ commutes with the $\R$-actions. 
\end{proof}

\section{Examples}\label{exemple}

We construct in this section a family of tilings of $\H$ of finite affine type, indexed by sequences on a finite alphabet. For uncountably many of them, the  tilings will be aperiodic and repetitive, the action on the associated hull will be free and minimal. A description of these actions in terms of  double-suspension is given.

\subsection{ Construction of the examples}  To construct such tilings
we will  use the hyperbolic Penrose's tiling described in section
\ref{hPT}, so we will keep the  notations of this section. Recall that
its stabilizer group under the action of $G$, is the group $\langle \
R\  \rangle$ generated by the  affine transformations  $R$. The main idea is to "decorate" this tiling in order to break its symmetry, the decoration will be coded by a sequence on a finite alphabet. By a decoration, we mean that we will substitute to each tile $t$ the  same polygon $t$ equipped with a color. We take the convention that two colored polygons  are the same if and only if the polygons are the same up to an affine map and they share the same color. By substituting each tile by a colored tile, we obtain a new tiling of finite affine type with a bigger number of prototiles. 

\noindent Notice that the coloration is not canonical. It also possible to do the same by substituting to a tile $t$, an unique finite family of  convex tiles $\{t_{i}\}_{i}$, like triangles,  such that the union of the $t_{i}$ is $t$ and the tiles $t_{i}$ overlaps only on their borders.  We choose the coloration only for presentation reasons.

Let $r$ be an integer bigger than $1$. We associate to each element of $\{1, \ldots, r\}$ an unique color. Let $P$ be the polygon defined in section \ref{hPT} to construct the Penrose's tiling. For an element $i$ of $\{1, \ldots, r\}$, we denote by $P_{i}$ the prototile $P$ colored in the color  $i$. To a sequence  $w= (w_k)_{k} \in $\{1, \ldots, r\}$^{\Z}$, we associate the $G$-finite-type tiling $\cP(w) $  built with the prototiles $P_{i}$ for $i$ in $\{1, \ldots, r\}$  and defined by:
$${\mathcal P}(w) = \{R^q \circ S^n ({P}_{w_{-q}}), \ n,q \in \Z\}.$$

\noindent Notice that the  stabilizer of this tiling is a subgroup of $\langle \ R\  \rangle$.

The set of sequences on $\{1, \ldots, r\}$ is the product space $\{1, \ldots, r\}^\Z$
which is   a Cantor set  for the product topology. There exists a natural homeomorphism on it called the $shift$. To a sequence
$(w_n)_{n\in \Z}$ the shift $\sigma$ associates the sequence
$(w_{n+1})_{n\in \Z}$.
Let $Z_{w}$ denote the closure of the orbit of $w$ by the action of the shift $\sigma$: $Z_{w}
=\overline{\{\sigma^n(w), \ n\in \Z\}}$. The set $Z_{w}$ is a compact metric space stable under the
action of $\sigma$.
\begin{remark}\label{rem-cont-colour}
The map $$Z_\omega\to X^G_{\cP(w)};\,\omega'\mapsto \cP(w')$$ is
continuous.
\end{remark}
\noindent Since $R. \cP(w)$ denotes the tiling image of $\cP(w)$ by $R$, we get the relation  
\begin{equation}\label{relation}
R.{\mathcal P}(w) = {\mathcal P}({\sigma(w)}).
\end{equation}
Thanks to  this, we obtain the following property:
\begin{lemma}\
\begin{itemize}
\item The sequence  $w$ is aperiodic for the shift-action, if and only
  if  $\cP(w)$ is  stable under no non-trivial affine map.
\item The dynamical system $(Z_{w}, \sigma)$ is minimal, if and only
  if  $(X^G_{\cP(w)}, G)$ is minimal.
\end{itemize}
\end{lemma}
 
\begin{proof} The first point comes from the relation \ref{relation}
  and from the fact that the stabilizer of $\cP(w)$ is a subgroup  of
  $\langle\  R\  \rangle$. The  last point comes from the
  characterization of minimal sequences: $(Z_{w}, \sigma)$ is minimal
  if and only if each words in $w$ appears infinitely many times  with
  uniformly bounded gap \cite{Qu}. This condition is equivalent
    to the repetitivity of   $\cP(w)$. \end{proof}

Recall that we have defined the group $\cN =\{z \mapsto z +t, \ t \in
\R\}$ and that $X^\cN_{\P}$ stands for  the closure (for the tiling
topology) of the $\cN$-orbit $\cN. \P$  in $X^G_{\cP}$ of the uncolored tiling
$\P$. Notice that the continuous action of $R$ on $X^G_{\cP}$
preserves the orbit $\cN. \P$  so  we get  an homeomorphism of $
X^\cN_{\P}$ that we denote also by $R$. We consider on the space
$X^\cN_{\P} \times Z_w \times \R^*_+$ equipped with the product
topology,  the homeomorphism $ \cR$ defined by $\cR( \cT, w', t) =
(R.\cT, \sigma(w'), t\slash 2)$. The quotient space $(X^\cN_{\P} \times Z_w \times \R^*_+)\slash \cR$, where the points in the same $\cR$ orbit are identified, is a compact space.

The affine group $G$ also  acts on the left on $X^\cN_{\P} \times Z_w \times \R^*_+$; where the action of an element $ g : z\mapsto az+b$ is given by the homeomorphism
$$( \cT, w', t) \mapsto (\cT+\frac{b}{at}, w', at) = g.(\cT,w',t). $$

It is straightforward to check that   the application $\cR$ commutes with this action, so this defines a $G$-continuous action on the quotient space $(X^\cN_{\P} \times Z_w \times \R^*_+)/ \cR$.

\begin{prop}\label{caracterisation} Let $w$ be an element of $\{1, \ldots, r\}^{\Z}$.
Then the map 
\begin{eqnarray*}
\Psi : G.\cP(w) & \to &  (X^\cN_{\P} \times Z_{w} \times \R^*_{+}) \slash \cR \\
g.\cP(w) & \mapsto & [g.(\P, w, 1)]
\end{eqnarray*}
where $[x]$ denotes the $\cR$-class of $x$, extends to  a conjugacy map between  $(X^G_{\P(w)}, G)$  and $\big( (X^\cN_{\P} \times Z_{w} \times \R^*_{+}) \slash \cR ,\  G \big).$
\end{prop}
\begin{proof} Let $\Phi$ be the transformation  $\cN. \P  \times Z_{w}\times
\R^*_{+} \to X^G_{\P(w)}$ defined by $$\Phi(\cP+\tau, w', t) =
R_{t}.(\cP(w')+\tau)$$ where $R_t$ denotes the map $z \mapsto tz$.
According to remark \ref{rem-cont-colour}, the application $\Phi$ is continuous for the tiling topology on
$\cN. \P$, so it extends by continuity to a continuous map from
$X^\cN_{\P} \times Z_{w} \times \R^*_{+} $ to $X^G_{\P(w)}$.  Thanks
to relation (\ref{relation}), we get  
$\Phi \circ \cR = \Phi $ on the dense subset  $\cN. \P\times Z_w \times \R^*_+$. 
Therefore the map $\Phi$ factorizes throught a continuous map
$\overline{\Phi} :  (X^\cN_{\P} \times Z_w \times \R^*_+)/ \cR \to
X^G_{\P(w)}$. Since the stabilizer of the tiling $\P(w)$ is a subgroup
of the one generated by the transformation $R$, and by relation (\ref{relation}),
the map $$\Psi : G\cdot \P(w) \to \big(X^\cN_{\P} \times Z_{w} \times
\R^*_{+}\big) / \cR;\,R_t\cdot  \P(w)+\tau\mapsto [\P+\tau/t,\omega,t]$$ 
is well defined. It is straightforward to check that $\Psi$ is
continuous, so it extends to a continuous map from $X^{G}_{\cP(w)}$ to
$ \big( X^\cN_{\P} \times Z_{w} \times \R^*_{+}\big) / \cR $ that we
denote again $\Psi$. Furthermore we have $\overline{\Phi} \circ \Psi =
Id$ on $G\cdot \P(w)$ and    $\Psi \circ \overline{\Phi} = Id$  on the  dense set $\pi(\cN. \P \times Z_w \times \R^*_+)$ where $\pi$ denotes the canonical projection $ X^\cN_{\P} \times Z_{w} \times \R^*_{+} \to \big( X^\cN_{\P} \times Z_{w} \times \R^*_{+} \big) \slash \cR$. Hence the map $\overline{\Phi}$ is an homeomorphism. The homeomorphism $\Psi$ obviously commutes with the action.  
\end{proof}

Notice that, $X^\cN_{\cP}$ is locally the Cartesian product of  a
Cantor set by an interval of $\R$. For $w \in \{1, \ldots, r\} ^\Z$,
$X^G_{\cP(w)}$ is locally homeomorphic  the product of a Cantor set by
an open subset (a slice) of $\R^*_{+} \times \R$ since the Cartesian product of two Cantor sets is again a Cantor set. The $G$-action maps slices onto slices.

\subsection{Ergodic properties of Penrose's tilings}\label{subsec-ergotic}





For a metric  space $X$ and a continuous action of a group $\Gamma$ on it, a $\Gamma$-{\em invariant measure} is a measure $\mu$ defined on the Borel $\sigma$-algebra of $X$ which is invariant under the action of $\Gamma$ i.e.: For any measurable set $B \subset X$ and any $g \in \Gamma$,    
$\mu (B.g)= \mu(B)$.
For instance, any group  $\Gamma$ acts on itself by right multiplication, there exists (up to a scalar) only  one measure invariant for this action: it is called the {\em Haar} measure.

Any action of an   amenable group $\Gamma$ (like $\Z$, $\R$ and all their extensions)   on a compact  metric space $X$  admits a finite invariant measure and in particular,  any homeomorphism $f$ of $X$ preserves a probability measure.
An {\em ergodic} invariant measure $\mu$ is such that every measurable functions constant along the orbits are $\mu$ almost surely constant. Every invariant measure is the sum of ergodic invariant measures \cite{Qu}. A conjugacy map sends the invariant measure to invariant measure and the ergodic measures to the ergodic measures.

In our case, the  group of affine transformations  $G$, is the extension of two groups isomorphic to $\R$, hence is amenable.  It is well known that the only invariant measures for the suspension action $\big(( X\times\R )\slash \cA_{f}, \R\big)$ are  locally the images through the canonical projection $\pi : X\times \R \to \big(X\times \R\big) \slash \cA_f$ of the measures $\mu \otimes \lambda$ where $\mu$ is a $f$-invariant measure on $X$ and $\lambda$ denotes the Lebesgue measure of $\R$. The proof is actually contained in property \ref{carmesureinv}.

It is  well known also that the map $o:  x \mapsto x+1$ on the dyadic
set of integers $\Omega$, admits only one  invariant  probability
measure:  the Haar probability measure on $\Omega$. Hence the
suspension of this action $\big( (\Omega \times \R) \slash \cA_{o}, \R
\big)$ admits only one invariant probability measure. By proposition
\ref{premsuspension}, $X_{\P}^\cN$ has only one invariant  probability
measure $\nu$.  Notice that the map $R$ preserves $X_{\P}^\cN$, and
since $R\cN R^{-1} = \cN$, the probability  $R_*\nu$ is 
$\cN$-invariant  and hence $R$ preserves
$\nu$. Nevertheless, the local product decomposition of $\nu$ is not
invariant by $R$, because $R$ divides by 2 the length of the intervals
of the $\cN$-orbit. So $R$ has to inflate the Haar measure on $\Omega$
by a factor $2$.

\begin{prop}\label{carmesureinv}
If  $w$ {is  an element in}  $\{1, \ldots, r\}^\Z$, then any 
 finite invariant measure of $\big( (X^\cN_{\P} \times Z_{w} \times \R^*_{+}) \slash \cR ,\  G \big)$ is locally  the image through the projection $X^\cN_{\P} \times Z_{w} \times \R^*_{+} \to  (X^\cN_{\P} \times Z_{w} \times \R^*_{+}) \slash \cR$ of a measure $\nu \otimes \mu \otimes m$ where 
\begin{itemize}
\item $\nu$ is the only invariant probability measure of $(X^\cN_{\P}, \R)$;
\item  $m$ is the Haar measure of $(\R^*_{+}, .)$;
\item  $\mu$ is a finite  invariant measure of $(Z_{w}, \sigma)$.
\end{itemize}
 \end{prop}
\begin{proof}
It is enough to prove this for an ergodic  finite $G$-invariant
measure $\overline {\theta}$  on the suspension $(X^\cN_{\P}\times
Z_{w} \times \R^*_{+} ) \slash \cR$. Since $\cR$ acts cocompactly on
$X^\cN_{\P}\times Z_{w} \times \R^*_{+}$,  $\overline{\theta}$ defines
a finite measure on a fundamental domain of  $\cR$, and the sum of all
the  images of this measure by iterates of $\cR$ and $\cR^{-1}$ defines a $\sigma$-finite measure $\theta$ on $X^\cN_{\P}\times Z_{w} \times \R^*_{+}$ which is $G$ and $\cR$-invariant. 

\noindent Let $\pi_{2} : X^\cN_{\P} \times Z_{w} \times \R^*_{+}\to
Z_{w} $ be the projection to the second coordinate, then
$\pi_{2*}\theta$ is a shift invariant measure on $Z_{w}$. The measure
$\theta$ can be disintegrated over $\pi_{2*}\theta=\mu$ by a family of
measures   $(\lambda_{w'})_{w'\in Z_w}$ defined for  $ \mu$-almost every $w' \in Z_{w}$ on $X^\cN_{\cP} \times \{w'\} \times \R^*_{+}$   such that 
$$\theta(B \times C  ) = \int_{C} \lambda_{w'}(B) d\mu(w'),$$
for any Borel sets $B \subset X^\cN_{\P}\times  \R^*_{+}$ and  $C \subset Z_{w}$.

\noindent The $G$-invariance of $\theta$ implies that the measures
$\lambda_{w'}$ are  $G$-invariant for almost all $w'$. 
The projection to the first coordinate $\pi_{1} :  X^{\cN}_{\P} \times
\{w'\} \times \R^*_{+} \to X^{\cN}_{\P}$ is   $\cN$-equivariant. The
measures $\pi_{1*}\lambda_{w'}$ are then $\cN$-invariant measures,
hence are proportional to $\nu$. Each measure $\lambda_{w'}$ can be
disintegrated over $\nu$ by a family of measures 
$(m_{x,w'})_{x\in X^\cN_{\P}  ,\,w'\in Z_w}$ on $\R^*_{+}$ defined for $\nu$-almost every $x \in X_{\P}^{\cN}$, so that 
 $$\lambda_{w'}(B \times \{w'\} \times I  ) = \int_{B} \int_{I} m_{x,w'} d\nu(x),$$  
for any Borel sets $B \subset X^\cN_{\P}$ and $I \subset \R^*_{+}$.
Each measure $\lambda_{w'}$ is invariant under the action of transformations of
the kind $z \mapsto az$ for $a\in\R^*_{+}$. It is then straightforward
to check that  the measures $m_{x,w'}$ are
multiplication-invariant for almost every $x$. By unicity of the Haar measure, there exists
a measurable positive function $(x,w') \mapsto h(x,w')$ defined almost
everywhere so that $m_{x,w'} = h(x,w') m$. The $\cN$-invariance of the
measures $\lambda_{w'}$ implies that the map $h$ is almost surely
constant along the $\cN$-orbits, and the $\cR$-invariance of $\theta$
implies that $h$ is almost surely constant along the
$\cR$-orbits. This defines then a measurable map on the quotient space
by $\cR$  which is $G$-invariant, the ergodicity of
$\overline{\theta}$  implies this map is almost surely constant.       
 \end{proof} 
 
Notice that an invariant measure on $X_{\cP(w)}^G$ can be decomposed locally into the product of a measure on a Cantor set by a measure along the leaves. Since the map $R$ does not preserve the transversal measure on $\Omega$ in $X^\cN_{\cP}$, the holonomy groupoid  of $X^G_{\cP(w)}$ does not preserve the transversal measure on the Cantor set.\\
The $G$-action is locally free and acts transitively on each leaf, so
each orbits inherits a hyperbolic 2-manifolds structure. Actually,
$X^G_{\cP(w)}$ can be equipped with a  continuous metric  with a
constant curvature $-1$ in restriction to the leaves. The invariant measures of $X^G_{\cP(w)}$ have then also a geometric interpretation in terms of {\em harmonic} measures, a notion introduced by L. Garnett in \cite{Ga}. 
  \begin{defn} A probability measure $\mu$ on $M$ is harmonic if
    $$\int_{M} \Delta f d\mu =0$$ for any continuous function $f$
  {with  restriction to  leaves of class $\mathcal{C}^2$}, where $\Delta$ denotes the  Laplace-Beltrami operator in restriction on each leaf.
 \end{defn} 
\noindent  Actually, it is shown in \cite{Pe} , that on $X^G_{\cP(w)}$ the notions of harmonic and invariant measures are the same and such measures can be described in terms of inverse limit of vectoriel cones.

\section{Transformation groupoids}


We gather this section with results on groupoids and their
$C^*$-algebras. Good material on this topic can be found in
\cite{re}. Let us fix first some notations. 
Let $\G$ be a locally compact groupoid, with base space $X$, range and
source
maps respectively  $r:\G\to X$ and  $s:\G\to X$. Recall that $X$ can
be viewed as a closed subset of $\G$ (the set of units). For any element $x$ of
$X$, we set $$\G^x=\{\gamma\in\G\text{ such that } r(\gamma)=x\}$$
and $$\G_x=\{\gamma\in\G\text{ such that } s(\gamma)=x\}.$$
Let us denote for any $\gamma$ in $\G$  by $L_\gamma:\G^{s(\gamma)}\to\G^{r(\gamma)}$ the left
translation by $\gamma$. 
Thourought this section, all the groupoids will be assumed  locally
compact and second countable.
Recall that a Haar system $\lambda$ for $\G$ is a family
$(\lambda^x)_{x\in X}$ of borelian measures on $\G$ such that
\begin{enumerate}
\item the support of $\lambda^x$ is $\G^x$;
\item for any $f$ in $C_c(\G)$, the map $X\to \C;\,x\to
  \int_{\G^x}fd\lambda^x$ is continuous;
\item ${L_\gamma}_*\lambda^{s(\ga)}=\lambda^{r(\ga)}$ for all $\ga$ in
    $\Ga$.
\end{enumerate}
Our prominent examples of groupoid will be {\it semi-direct product groupoid}:
let $H$ be a locally compact group acting on a locally compact space
$X$. The  semi-direct product groupoid $X\rtimes H$ of $X$ by $H$ is
defined by 
\begin{itemize}
\item $X\times H$
as a topological space ;
\item the base space  is  $X$ and  the structure maps are $r:X\rtimes
H\to X;\, (x,h)\mapsto x$ and $s:X\rtimes
H\to X;\, (x,h)\mapsto h^{-1}x$ ;
\item  the product is 
$(x,h)\cdot(h^{-1}x,h')=(x,hh')$ for  $x$ in
$X$ and $h$ and $h'$ in $H$.
\end{itemize}
 Let $\mu$ be a left Haar mesure on
$H$. Then the  groupoid $X\rtimes H$ is equipped with a Haar system
$\lambda^\mu=(\lambda^\mu_{x})_{x\in X}$ 
  given for any $f$ in $C_c(X\times H)$ and any $x$ in
  $X$ by
$\lambda^\mu_{x}(f)=\int_{H}f(x,h)d\mu(h)$.
\subsection{Suspension of a
  groupoid}\label{subsec-groupoid-suspension}
Recall that any automorphism $\alpha$ of a groupoid $\G$ induces a
homeomorphism of its base space {$X$} that we shall denote by $\alpha_X$.
\begin{definition}
Let  $\G$ be a  groupoid  with base space $X$ equipped
with a Haar system $\lambda=(\lambda^x)_{x\in X}$. A groupoid automorphism 
$\alpha:\G\to\G$ is said to preserve the Haar system $\lambda$ if there exists a
continuous function  $\rho_\alpha:\G\to\R^+$
such that  for any $x$ in $X$ the measures $\alpha_*\lambda^x$ and
  $\lambda^{\alpha(x)}$ on
  $\G^{\alpha(x)}$ are in the same class and   $\rho_\alpha$
    restricted  to 
      $\G^{\alpha(x)}$ is   $\frac{d\alpha_*\lambda^x}{d\lambda^{\alpha(x)}}$. The map $\rho_\alpha$ is called the density of
      $\alpha$.
\end{definition}

\begin{remark}\label{rem-density}Let  $\G$ be a  groupoid  with base
  space $X$ and  Haar system $\lambda=(\lambda^x)_{x\in X}$ and let
  $\alpha:\G\to\G$ be  an automorphism of
groupoid preserving the Haar system $\lambda$.
\begin{enumerate}
\item   Since
$L_{\gamma}\circ\alpha=\alpha\circ L_{\alpha^{-1}(\gamma)}$ for any $\gamma$
  in $\G$, we get that
  $$L_{\gamma,*}\alpha_*\lambda^{\alpha^{-1}(s(\gamma))}=\alpha_*L_{\alpha^{-1}(\gamma),*}
  \lambda^{s(\alpha^{-1}(\gamma))}=\alpha_*\lambda^{r(\alpha^{-1}(\gamma))}.$$ Since 
  $L_{\gamma,*}\alpha_*\lambda^{\alpha^{-1}(s(\gamma))}$ is a measure  on  $\G^{r(\gamma)}$
absolutly
  continuous with respect to
  $L_{\gamma,*}\lambda^{s(\gamma)}=\lambda^{r(\gamma)}$ with density
  $\rho_\alpha\circ L_{\gamma^{-1}}$  we see that 
  $\rho_\alpha\circ
  L_{\gamma^{-1}}$ and $\rho_\alpha$ coincide on $\G^{r(\ga)}$.  In particular $\rho_\alpha$ is
  constant on $\G_x$ for any $x$ in $X$.
  
\item The automorphism of groupoid $\alpha^{-1}:\G\to\G$ also preserves
  the Haar system $\lambda$ and
  $\rho_{\alpha^{-1}}=1/\rho_{\alpha}\circ\alpha$.
\end{enumerate}
\end{remark}
\begin{definition}
Let  $\G$ be a  groupoid  with base space $X$, range
and source
map $r$ and $s$  and let  
$\alpha:\G\to\G$ be a groupoid automorphism. Using the notations of
section \ref{subsec-hull} the suspension of the groupoid  $\G$
respectively to  $\alpha$ is the groupoid $\G_\alpha\defi
(\G\times\R)/\cA_\alpha$ with base space  $X_{\alpha}\defi(X\times\R)/\cA_{\alpha_X}$. For any $\gamma$ in
$\G$ and $t$ in $\R$, let us denote by $[\gamma,t]$ the class of
$(\gamma,t)$ in $\G_\alpha$.
\begin{itemize}
\item The range map $r_\alpha$ and the source map $s_\alpha$ are
  defined in the following way:
\begin{itemize}
\item $r_\alpha([\gamma,t])=[r(\gamma),t]$ for every $\gamma$ in
$\G$ and $t$ in $\R$;
\item $s_\alpha([\gamma,t])=[s(\gamma),t]$ for every $\gamma$ in
$\G$ and $t$ in $\R$;
\end{itemize}
\item Let $\gamma$ and $\gamma'$ be elements of $\G$ such that
  $s(\gamma)=r(\gamma')$ and let $t$ be in
  $\R$, then $[\gamma,t]\circ[\gamma',t]=[\gamma\circ\gamma',t]$;
\item $[\gamma,t]^{-1}=[\gamma^{-1},t]$.
\end{itemize}
There is an action of $\R$ on $\G_\alpha$ by automorphisms given for $s$ in $\R$ and
$[\gamma,t]$ in $\G_\alpha$ by $s\cdot [\gamma,t]=[\gamma,t+s]$.

\end{definition}


\begin{lemma}\label{rem-pr} Let  $\G$ be a  groupoid  with base space $X$ equipped
with a Haar system $\lambda=(\lambda^x)_{x\in X}$ and  let
$\alpha:\G\to\G$ be an automorphism 
 preserving the Haar System $\lambda$. Let us assume that
 $\rho_\alpha=\rho_\alpha\circ\alpha$. Then $\G_\alpha$ admits a Haar
 system
$\lambda_\alpha=\left(\lambda_\alpha^{[x,t]}\right)_{[x,t]\in X_{\alpha}}$
  given for any $[x,t]$ in $X_{\alpha}$ and any continuous
  fonction $f$ in $C_c\left(\G_\alpha^{[x,t]}\right)$ by 
$$\lambda^{[x,t]}_\alpha(f)=\int_{\G^x}\rho_\alpha(\gamma)^{-t}f([\gamma,t])d\lambda^x(\gamma).$$
\end{lemma}
\begin{proof}\
\begin{itemize}
\item Let us prove first that the definition of $\lambda^{[x,t]}_\alpha(f)$ for 
 $[x,t]$ in  $X_{\alpha}$ and $f$ in
 $C_c\left(\G_\alpha^{[x,t]}\right)$ makes sense.
\begin{eqnarray*}
\int_{\G^x}\rho_\alpha(\gamma)^{-t}f([\gamma,t])d\lambda^x(\gamma)&=&\int_{\G^x}\rho_\alpha(\alpha(\gamma))^{-t}f([\alpha^{-1}(\alpha(\gamma),t])d\lambda^x(\gamma)\\
&=&\int_{\G^{\alpha(x)}}\rho_\alpha(\gamma)^{-t+1}f([\alpha^{-1}(\gamma),t])d\lambda^{\alpha(x)}(\gamma)\\
&=&\int_{\G^{\alpha(x)}}\rho_\alpha(\gamma)^{-t+1}f([\gamma,t-1])d\lambda^{\alpha(x)}(\gamma).\\
\end{eqnarray*}\item It is clear that the continuity condition is
fullfilled. Let us show then  that $(\lambda^{[x,t]})_{[x,t]\in X_{\alpha}}$ is a Haar system. Let $\ga'$ be an element of $\G$, let $t$ be a real
number  and let $f$ be a
function in $C_c\left(\G_\alpha^{[r(\gamma'),t]}\right)$. Then
\begin{eqnarray*}
\lambda_\alpha^{[r(\gamma'),t]}(f)&=&\int_{\G^{r(\gamma')}}\rho_\alpha(\gamma)^{-t}f([\gamma,t])d\lambda^{r(\gamma')}(\gamma)\\
&=&\int_{\G^{s(\gamma')}}\rho_\alpha(\gamma'\cdot\gamma)^{-t}f([\gamma'\cdot\gamma,t])d\lambda^{s(\gamma')}(\gamma)\\
&=&\int_{\G^{s(\gamma')}}\rho_\alpha(\gamma)^{-t}f([\gamma'\cdot\gamma,t])d\lambda^{s(\gamma')}(\gamma)\\
&=&\lambda_\alpha^{[s(\gamma'),t]}(f\circ
L_{[\gamma',t]}),\end{eqnarray*}
where the third equality holds in view of remark \ref{rem-density}.
\end{itemize}
\end{proof}
\subsection{$C^*$-algebra of a suspension groupoid}\label{subsec-suspension}
Let us recall first the construction of the reduced
$C^*$-algebra $C_r^*(\G,\lambda)$  associated to  a 
groupoid $\G$ with base $X$ and Haar
system $\lambda=(\lambda^x)_{x\in X}$. Let $\L^2(\G,\lambda)$ be the
$C_0(X)$-Hilbert completion of $C_c(\G)$ equipped with the
$C_0(X)$-valued scalar product 
$$\langle
\phi,\phi'\rangle(x)=\int_{\G^x}\bar{\phi}(\gamma^{-1})\phi'(\gamma^{-1})d\lambda^x(\gamma)$$
for $\phi$ and $\phi'$ in $C_c(\G)$ and $x$ in $X$, i.e the completion
of $C_c(\G)$ with respect to the norm $\|\phi\|=\sup_{x\in X}\langle
\phi,\phi\rangle^{1/2}$. 
The $C_0(X)$-module structure on  $C_c(\G)$ extends to
$\L^2(\G,\lambda)$ and $\langle
\bullet,\bullet\rangle$ extends to a $C_0(X)$-valued scalar product on
$\L^2(\G,\lambda)$. Recall that an operator
$T:\L^2(\G,\lambda)\to\L^2(\G,\lambda)$ is called adjointable if there
exists an operator $T^*:\L^2(\G,\lambda)\to\L^2(\G,\lambda)$, called
the adjoint of $T$ such that $$\langle
T^*\phi,\phi'\rangle=\langle
\phi,T\phi'\rangle$$ for all $\phi$ and $\phi'$ in
$\L^2(\G,\lambda)$. Notice that the adjoint, when it exists is unique
and that operator that admits an adjoint are automatically
$C_0(X)$-linear and continuous. The set of adjointable operators on
$\L^2(\G,\lambda)$ is then a $C^*$-algebra with respect to the
operator norm.
Then any $f$ in
$C_c(\G)$ acts as an adjointable  operator on
$\L^2(\G,\lambda)$ by convolution
$$f\cdot\phi(\gamma)=\int_{\G^{r(\gamma)}}f(\gamma')\phi(\gamma'^{-1}\gamma)d\lambda^{r(\gamma)}(\gamma')$$
where  $\phi$ is in $C_c(\G)$, the  adjoint of this operator being given by the action of
$f^*:\gamma\mapsto \bar{f}(\gamma^{-1})$.
The convolution product provides a structure of involutive algebra on
$C_c(\G)$ and using the action defined above, this algebra can be viewed as a subalgebra of the
$C^*$-algebra of adjointable operators of $\L^2(\G,\lambda)$. The
reduced $C^*$-algebra $C_r^*(\G,\lambda)$ is then the closure of
$C_c(\G)$ in the $C^*$-algebra of adjointable operators of
$\L^2(\G,\lambda)$. Namely, if we define for $x$ in $X$ the measure on
$\G_x$ by $\lambda_x(\phi)=\int_{\G^x}\phi(\gamma^{-1})d\lambda^x(\gamma)$
for any $\phi$ in $C_c(\G_x)$, then $\L^2(\G,\lambda)$ is a continuous
field of Hilbert spaces over $X$ with
fiber $\L^2(\G_x,\lambda_x)$ at $x$ in $X$. The  corresponding $C_0(X)$-structure on
$C_r^*(\G,\lambda)$ is then given for $h$ in $C_0(X)$ by the multiplication by
$h\circ s$.
\begin{ex} Let $H$ be a locally compact group acting on a locally
  compact space $X$, and consider the semi-direct product  groupoid 
 $X\rtimes H$  equipped with a Haar system
arising from the Haar measure on $H$.
Then  $C^*_r(X\times H,\lambda^\mu)$ is the usual reduced crossed product
$C_0(X)\rtimes_r H$. 
\end{ex}

 Let us  denote for any $x$ in $X$ by $\nu_x$ the representation of
$C_r^*(\G,\lambda)$ on the fiber $\L^2(\G_x,\lambda_x)$. Then for any
$f$ in $C_r^*(\G,\lambda)$, we get that
$\|f\|_{C_r^*(\G,\lambda)}=\sup_{x\in X}\|\nu_x(f)\|$.

\begin{lemma}\label{lem-mt}
 Let  $\G$ be a locally compact groupoid  with base space $X$ equipped
with a Haar system $\lambda=(\lambda^x)_{x\in X}$ and  let
$\alpha:\G\to\G$ be an automorphism 
 preserving the Haar System $\lambda$. Let us define the continuous
 map
 $\rho_\alpha':\G\to\R;\,\gamma\mapsto\rho_\alpha(\gamma^{-1})$. Then
 there exists a unique 
 automorphism $\tilde{\alpha}$  of the $C^*$-algebra
 $C_r^*(\G,\lambda)$ such that for every $f$ in $C_c(\G)$ we have
 $\tilde{\alpha}(f)=(\rho_\alpha'\rho_\alpha)^{1/2}f\circ\alpha^{-1}$.
\end{lemma}
\begin{proof}
The map $C_c(\G)\to C_c(\G);\,
\phi\mapsto\rho_\alpha'^{1/2}\phi\circ\alpha^{-1}$ extends uniquely to
a continuous linear and invertible map
$W:\L^2(\G,\lambda)\to\L^2(\G,\lambda)$ such that $$\langle
W\cdot\phi,W\cdot\phi\rangle(x)=\langle
\phi,\phi\rangle(\alpha^{-1}(x)),$$ for all $x$ in $X$. Its  inverse
$W^{-1}$ is defined by 
$W^{-1}(\phi)=({\rho'}_\alpha\circ\alpha)^{-1/2}\phi\circ\alpha$
for all $\phi$ in $C_c(\G)$. Let us define
$$\tilde{\alpha}:C_r^*(\G,\lambda)\to C_r^*(\G,\lambda);\, x\mapsto
W\cdot x\cdot W^{-1}.$$ Then $W\cdot f\cdot
W^{-1}=(\rho_\alpha'\rho_\alpha)^{1/2}f\circ\alpha^{-1}$ for all $f$
in $C_c(\G)$.
\end{proof} 
Recall that if $A$ is a $C^*$-algebra and if $\beta$ is an
automorphism of $A$ then the mapping torus of $A$ is the $C^*$-algebra
$$A_\beta=\{f\in C([0,1],A)\text{ such that }\beta(f(1))=f(0)\}.$$
Namely, the mapping torus $A_\beta$ can be viewed as the algebra of
 continuous function   $h:\R\to A$ such that
$h(t)=\beta(h(t+1))$ for all $t$ in $\R$.
In this picture, there is an action  of $\R$ on $A_\beta$ by
translations defined for $t$ in $\R$ and $f$ in $A_\beta$ by $$t\cdot f(s)=f(s-t)$$ for any $s$ in $\R$. Translations
then  define a strongly continuous action by automorphisms $\widehat{\beta}$ of $\R$ on
$A_\beta$. By the mapping torus isomorphism, we have a natural Morita
equivalence between $A\rtimes_\beta\Z$ and $A\rtimes_{\widehat{\beta}}\R$.

Let  $\alpha$ be an automorphism of a groupoid $\G$ preserving a Haar
system $\lambda$ and with density $\rho_\alpha$. 
For a function $f$ in $C_c(\G_\alpha)$, we define $\hat{f}$ in
$C_c([0,1]\times\G)\subset C([0,1],C_r^*(\G,\lambda))$ by
$\hat{f}(t,\gamma)=\rho_\alpha^{-t/2}(\gamma){\rho'}_\alpha^{-t/2}(\gamma)
f([\gamma,t])$. We can check easily that $\hat{f}$ belongs to the
mapping torus $C_r^*(\G,\lambda)_{\tilde{\alpha}}$.

\begin{proposition}\label{prop-iso}
 Let  $\G$ be a locally compact groupoid  with base space $X$ equipped
with a Haar system $\lambda=(\lambda^x)_{x\in X}$ and  let
$\alpha:\G\to\G$ be an automorphism 
 preserving the Haar System $\lambda$ such that
 $\rho_\alpha\circ\alpha=\rho_\alpha$. Then there is  an unique automorphism of
 $C^*$-algebras $$\Lambda_\alpha:C_r^*(\G_\alpha,\lambda_\alpha)\to
 C_r^*(\G,\lambda)_{\tilde{\alpha}}$$ such that
 $\Lambda_\alpha(f)=\hat{f}$ for any $f$ in $C_c(\G_\alpha)$.
\end{proposition}

\begin{proof}Let $f$ be a function  of $C_c(\G_\alpha)$. Then
\begin{eqnarray*}
\|\hat{f}\|_{
  C_r^*(\G,\lambda)_{\tilde{\alpha}}}&=&\sup_{t\in[0,1]}\|\hat{f}(t,\bullet))\|_{
  C_r^*(\G,\lambda)}\\
&=& \sup_{t\in[0,1],\,x\in X}\|\nu_x(\hat{f}(t,\bullet))\|
\end{eqnarray*}
On the other hand,
$$\displaystyle\|f\|_{C_r^*(\G_\alpha,\lambda_\alpha)}=\sup_{t\in[0,1],\,x\in
    X}\|\nu_{[x,t]}(f)\|,$$ where $\nu_{[x,t]}$ is the representation
  of $C_r^*(\G_\alpha,\lambda_\alpha)$ on the fiber
    $\L^2(\G_{\alpha,[x,t]},\lambda_{\alpha,[x,t]})$ at 
    $[x,t]\in (X\times\R)/\cA_{\alpha_X}$.
If we define for $t$ in $[0,1]$ the map
$\pi_t:\G\to\G_\alpha:\,\gamma\mapsto [\gamma,t]$, then 
$$C_c(\G_{[x,t]})\to C_c(\G_x):\,\phi\mapsto{\rho'}_\alpha^{-t/2}\phi\circ\pi_t$$
  extends to an isometry
  $W_t:\L^2(\G_{\alpha,[x,t]},\lambda_{\alpha,[x,t]})\to\L^2(\G_x,\lambda_x)$ and
 $W_t$ conjugate $\nu_{[x,t]}(f)$ and $\nu_x(\hat{f}(t,\bullet))$. Thus
 $\|\hat{f}\|_{
  C_r^*(\G,\lambda)_{\tilde{\alpha}}}=\|f\|_{C_r^*(\G_\alpha,\lambda_\alpha)}$ and 
 $$C_c(\G_\alpha)\to
 C_r^*(\G,\lambda)_{\tilde{\alpha}};\,f\mapsto\hat{f}$$ extends to a
monomorphism   $\Lambda_\alpha:C_r^*(\G_\alpha,\lambda_\alpha)\to
 C_r^*(\G,\lambda)_{\tilde{\alpha}}$. The set $$\mathcal{A}_\alpha=\{h\in
 C_c([0,1]\times\G)\text{ such that
 }h(1,\alpha(\gamma))={\rho'}_\alpha^{t/2}\rho_\alpha^{t/2}h(0,\gamma)\text{ for all }\gamma\in\G\}$$ is
 dense in $C_r^*(\G,\lambda)_{\tilde{\alpha}}$. Let us define for an
 element $h$
 of $\mathcal{A}_\alpha$ the map $\tilde{h}:\G\times\R\to\C$
 as the unique  map such that
\begin{itemize}
\item
  $\tilde{h}(\gamma,t)={\rho'}_\alpha^{-t/2}\rho_\alpha^{-t/2}h(t,\gamma)$ for all $\gamma$ in $\G$ and $t$ in $[0,1]$;
\item $h(\alpha(\gamma),t)=h(\gamma,t+1)$  for all $\gamma$ in $\G$
  and $t$ in $\R$.
\end{itemize}
Then $\tilde{h}$ defines a continuous map of $C_c(\G_\alpha)$ whose image under $\Lambda_\alpha$ is $h$. Hence $\Lambda_\alpha$ has dense range in  $C_r^*(\G,\lambda)_{\tilde{\alpha}}$ and thus is surjective.
\end{proof}
\begin{remark}\label{rem-action}
With  the notations of  above  proposition, let us define for a real $s$ the automorphism of groupoid
$\theta_s:\G_\alpha\to\G_\alpha;\,[\gamma,t]\mapsto[\gamma,s+t]$. Then
$\theta_s$ is preserving the Haar system
$\lambda_\alpha=(\lambda^{[x,t]})_{[x,t]\in X_{\alpha}}$ with
density $$\G_\alpha\to\R;\,[\gamma,t]\mapsto \rho_\alpha(\gamma)^{s}.$$ We
obtain from lemma \ref{lem-mt} an automorphism $\tilde{\theta}_s$ of
$C_r^*(\G_\alpha,\lambda_\alpha)$ which gives rise to a strongly
continuous action of $\R$ on $C_r^*(\G_\alpha,\lambda_\alpha)$ by
automorphism. The isomorphism $$\Lambda_\alpha:C_r^*(\G_\alpha,\lambda_\alpha)\to
 C_r^*(\G,\lambda)_{\tilde{\alpha}}$$ of proposition \ref{prop-iso} is
 then $\R$-equivariant, where 
the action of
$\R$ on $C_r^*(\G,\lambda)_{\tilde{\alpha}}$
is  the action $\widehat{\tilde{\alpha}}$ associated to a mapping torus.
\end{remark}

\section{The dynamic of the uncolored Penrose tiling under
  translations}\label{sec-penrose-uncolor}

As we have seen before, the closure   $X^\cN_{\mathcal{P}}$ of
$\cN\cdot{\mathcal{P}}$ for the tiling topology  is the suspension 
$(\Omega\times\R)\slash \cA_o$ of the odometer homeomorphism
$o:\Omega\to\Omega;\,x\mapsto x+1$, where $\Omega$ is  the
dyadic completion  of the integers. The $\R$-algebra
$C({(\Omega\times\R)}\slash {\cA_o})$ is then the mapping torus algebra of
$C(\Omega)$ with respect to automorphism induced by $o$. In
consequence, the crossed product algebras
$C(X^\cN_{\mathcal{P}})\rtimes\R$ and $C(\Omega)\rtimes\Z$ are Morita
equivalent. The purpose of this section is to recall the explicit
description of the isomorphism
$C(\Omega)\rtimes\Z\stackrel{\cong}{\to}
C(X^\cN_{\mathcal{P}})\rtimes\R $ arising from this
Morita equivalence.

For this, let us define on $C_c(\Omega\times\R)$ the
$C(\Omega)\rtimes\Z$-valued inner product
$$\langle\xi,\,\xi'\rangle(\omega,k)=\int_\R\bar{\xi}(\omega,s)\xi'(\omega-k,s+k)ds$$for
$\xi$ and $\xi'$ in $C_c(\Omega\times\R)$ and $(\omega,k)$ in
$(\Omega\times\R)$. This inner product is positive and gives rise to a
right $C(\Omega)\rtimes\Z$-Hilbert module $\E$, the action of
$C(\Omega)\rtimes\Z$ being given for $h$ in $C_c(\Omega\times\Z)$ and $\xi$ in 
$C_c(\Omega\times\R)$ by
$$\xi\cdot
h(\omega,t)=\sum_{n\in\Z}\xi(n+\omega,t-n)h(n+\omega,n)$$ for
 $(\omega,k)$ in
$\Omega\times\R$. The right $C(\Omega)\rtimes\Z$-Hilbert module $\E$
is also equipped with a left action of
$C\left({(\Omega\times\R)}\slash {\cA_o}\right)\rtimes\R$ given for $f$ in
$C_c\left({(\Omega\times\R)}\slash {\cA_o}\times\R\right)$ and $\xi$ in 
$C_c(\Omega\times\R)$ by
$$f\cdot\xi(\omega,t)=\int_\R f([\omega,t],s)\xi(\omega,t-s)ds$$for
 $(\omega,k)$ in
$\Omega\times\R$. We get in this way a
$C\left({(\Omega\times\R)}\slash {\cA_o}\right)\rtimes\R-C(\Omega)\rtimes\Z$
imprimitivity bimodule which implements the Morita equivalence we are
looking for. Actually, there is an isomorphism of right
$C(\Omega)\rtimes\Z$-Hilbert module
$$\Psi:\E\to {L}^2([0,1])\otimes C(\Omega)\rtimes\Z$$
 defined in  a unique way  by
$\Psi(g)=g\otimes u$ for  $g$ in $C_c(\R)$ supported in $(0,1)$, where
$u$ is the unitary of $C(\Omega)\rtimes\Z$ corresponding to the
positive generator of $\Z$.
Using the right
$C\left({(\Omega\times\R)}\slash {\cA_o}\right)\rtimes\R$-module structure of
the  $C\left({(\Omega\times\R)}\slash {\cA_o}\right)\rtimes\R-C(\Omega)\rtimes\Z$
imprimitivity bimodule $\E$ and the isomorphism $\Psi$, we get an
isomorphism
\begin{equation}\label{equ-Morita}
C\left({(\Omega\times\R)}\slash {\cA_o}\right)\rtimes\R\stackrel{\cong}{\to}\K({L}^2([0,1]))\otimes
C(\Omega)\rtimes\Z.
\end{equation}
This isomorphism can be described as follows. Let us define for $f$
and $g$ in  ${L}^2([0,1])$ the rank one operator
$$\Theta_{f,g}:{L}^2([0,1])\to{L}^2([0,1]);\,h\mapsto f\langle
g,h\rangle.$$ We define for $\xi$ and $\xi'$ in $C_c(\Omega\times\R)$
the  continuous function of
$C_c\left({(\Omega\times\R)}\slash {\cA_o}\times\R\right)$ 
$$\Theta^\Omega_{\xi,\xi'}([\omega,s],t)=\sum_{k\in\Z}\xi(\omega+k,s-k)\bar{\xi'}(\omega+k,s-t-k)
$$  
  for all $\omega$ in $\Omega$ and $s$ and $t$ in $\R$.
It is straightforward to check that $\Theta^\Omega_{\xi,\xi'}$ is well
defined and that
$$\Theta^\Omega_{\xi,\xi'}\cdot\eta=\xi\langle\xi',\eta\rangle$$for
all $\eta$ in $C_c(\Omega\times\R)$.
If we set for $f$ and $g$ in $C_c(\R)$ with support in $(0,1)$ and for $\phi$ in
$C(\Omega)$,
$\xi=1\otimes f$, $\xi'=\phi\otimes g$ and
$\xi'':\Omega\times\R\to\R;(\omega,t)\mapsto g(t+1)$, then the image
of $\Theta^\Omega_{\xi,\xi'}$ under the isomorphism of equation
(\ref{equ-Morita}) is $\Theta_{f,g}\otimes\phi\in\K({L}^2([0,1]))\otimes
C(\Omega)\rtimes\Z$ and moreover,
\begin{equation}\label{equ-th1}\Theta^\Omega_{\xi,\xi'}([\omega,s],t)=\sum_{k\in\Z}f(s-k)\bar{\phi}(\omega+k)\bar{g}(s-t-k).\end{equation}
 The image
of $\Theta^\Omega_{\xi,\xi''}$ under the isomorphism of equation
(\ref{equ-Morita}) is $\Theta_{f,g}\otimes u \in\K({L}^2([0,1]))\otimes
C(\Omega)\rtimes\Z$ and moreover,
\begin{equation}\label{equ-th2}\Theta^\Omega_{\xi,\xi''}([\omega,s],t)=\sum_{k\in\Z}f(s-k)\bar{g}(s+1-t-k).\end{equation}
Let us define  the automorphism $\alpha$  of the groupoid
${(\Omega\times\R)}\slash {\cA_o}\rtimes\R$ in the following way
\begin{itemize}
\item $\alpha([\omega,s],t)=([\omega/2,s/2],t/2)$ if $\omega$ is even;
\item $\alpha([\omega,s],t)=([(\omega+1)/2,(s+1)/2],t/2)$ if $\omega$
  is odd.
\end{itemize}
Notice that $\alpha^{-1}([\omega,s],t)=([2\omega,2s],2t)$ for all
$\omega$ in $\Omega$ and $s$ and $t$ in $\R$.
Then $\alpha$ preserves the Haar system of
${(\Omega\times\R)}\slash {\cA_o}\rtimes\R$ arising from the  Haar mesure on
$\R$ and has constant density $\rho_\alpha=2$. Hence  according to lemma
\ref{lem-mt}, the automorphism of groupoid $\alpha$ induces an
automorphism $\tilde{\alpha}$ of $C^*$-algebra
$C\left({(\Omega\times\R)}\slash {\cA_o}\right)\rtimes\R$ such that 
$\tilde{\alpha}(h)=2h\circ\alpha^{-1}$ for all $h$ in
$C\left({(\Omega\times\R)}\slash {\cA_o}\times\R\right)$.
We are now in position to describe how $\tilde{\alpha}$ is transported under
the isomorphism of equation (\ref{equ-Morita}) to an automorphism $\Upsilon$ of 
$\K({L}^2([0,1]))\otimes
C(\Omega)\rtimes\Z$. With $\xi$, $\xi'$ and $\xi''$ as defined above,
\begin{eqnarray}\label{equ-th7}
\nonumber\tilde{\alpha}(\Theta^\Omega_{\xi,\xi'})([\omega,s],t)&=&2\Theta^\Omega_{\xi,\xi'}([2\omega,2s],2t)\\
\nonumber &=&2\sum_{k\in\Z}f(2s-k)\bar{\phi}(2\omega+k)\bar{g}(2s-2t-k)\\
 &=&2\sum_{k\in\Z}f(2s-2k)\bar{\phi}(2\omega+2k)\bar{g}(2s-2t-2k)+\\
\nonumber &&2\sum_{k\in\Z}f(2s-2k-1)\bar{\phi}(2\omega+2k-1)\bar{g}(2s-2t-2k-1)
\end{eqnarray}
and

\begin{eqnarray}\label{equ-th8}
\nonumber\tilde{\alpha}(\Theta^\Omega_{\xi,\xi''})([\omega,s],t)&=&2\Theta^\Omega_{\xi,\xi''}([2\omega,2s],2t)\\
\nonumber &=&2\sum_{k\in\Z}f(2s-k)\bar{g}(2s+1-2t-k)\\
&=&2\sum_{k\in\Z}f(2s-2k)\bar{g}(2s+1-2t-2k)+\\
\nonumber&&\quad\quad 2\sum_{k\in\Z}f(2s-2k-1)\bar{g}(2s-2t-2k).
\end{eqnarray}
To complete the description of the automorphism $\Upsilon$ of $\K({L}^2([0,1]))\otimes
C(\Omega)\rtimes\Z$ corresponding to $\tilde{\alpha}$, we need to
introduce some further  notations.
We define the partial isometries $U_0,U_1$ and $V$ of ${L}^2([0,1])$
by 
\begin{itemize}
\item $U_0 f(t)=\sqrt{2}f(2t)$ if $t\in[0,1/2]$ and $U_0 f(t)=0$ otherwise;
\item $U_1 f(t)=\sqrt{2}f(2t-1)$ if $t\in[1/2,1]$ and $U_1
  f(t)=0$ otherwise;
 \item $V f(t)=f(t+1/2)$ if $t\in[0,1/2]$ and $V
  f(t)=0$ otherwise,
\end{itemize} for $f$ in $C([0,1])$.
Let use define also the endomorphisms $W_0$ and $W_1$ of the
$C^*$-algebra  $C(\Omega)$ 
by $W_0\phi(\omega)=\phi(2\omega)$ and
$W_1\phi(\omega)=\phi(2\omega+1)$, for $\phi$ in $C(\Omega)$ and
$\omega$ in $\Omega$.
Using this notations, equations (\ref{equ-th7}) and  (\ref{equ-th8}) can
be rewriten as
$$\tilde{\alpha}(\Theta^\Omega_{\xi,\xi'})([\omega,s],t)=\sum_{k\in\Z}U_0
f(s-k)W_0\bar{\phi}(\omega+k)U_0\bar{g}(s-t-k)+\sum_{k\in\Z}U_1
f(s-k)W_1\bar{\phi}(\omega+k)U_1\bar{g}(s-t-k)$$
and
$$\tilde{\alpha}(\Theta^\Omega_{\xi,\xi''})([\omega,s],t)=\sum_{k\in\Z}U_0
f(s-k)U_1\bar{g}(s-t-k+1)+\sum_{k\in\Z}U_1
f(s-k)U_0\bar{g}(s-t-k).$$ Thus, in view of equations (\ref{equ-th1})
and  (\ref{equ-th2}), we get that
$$\Upsilon(\Theta_{f,g}\otimes\phi)=\Theta_{U_0f,U_0g}\otimes
W_0\phi+\Theta_{U_1f,U_1g}\otimes W_1\phi$$ and
$$\Upsilon(\Theta_{f,g}\otimes u)=\Theta_{U_0f,U_1g}\otimes u
+\Theta_{U_1f,U_0g}\otimes 1.$$ From this we deduce
$$\Upsilon(k\otimes\phi)=U_0\cdot k\cdot U_0^*\otimes W_0\phi+U_1\cdot
k\cdot U_1^*\otimes W_1\phi$$
and 
\begin{eqnarray*}
\Upsilon(k\otimes u)&=& U_0\cdot k\cdot U_1^* \otimes u+U_1\cdot k\cdot U_0^* \otimes 1\\
&=& U_0\cdot k\cdot U_0^* \cdot V \otimes u+U_1\cdot k\cdot U_1^*\cdot
V^* \otimes 1\\
&=& (U_0\cdot k\cdot U_0^*+U_1\cdot k\cdot U_1^*)\cdot(V \otimes u+\cdot
V^* \otimes 1)
\end{eqnarray*}
where the second  equality holds since $V^*\cdot U_0=U_1$ and $V\cdot
U_1=U_0$ and the third holds since $V^*U_1=VU_0=0$.
In consequence, if we extends $\Upsilon$ to the multiplier algebra
of $\K({L}^2([0,1]))\otimes
C(\Omega)\rtimes\Z$, we finally obtain that the automorphism  $\Upsilon$ is the unique
homomorphism of $C^*$-algebra such that 
$$\Upsilon(k\otimes\phi)=U_0^*\cdot k\cdot U_0\otimes W_0\phi+U_1^*\cdot
k\cdot U_1\otimes W_1\phi$$ and
\begin{equation}\label{equ-u}\Upsilon(1\otimes u)=V\otimes u+V^*\otimes 1,\end{equation}
 where $k$ is in
$\K({L}^2([0,1]))$, $\phi$ is in $C(\Omega)$ and $1\otimes u$ and
$V\otimes u+V^*\otimes 1$ are viewed as multipliers of  $\K({L}^2([0,1]))\otimes
C(\Omega)\rtimes\Z$.

The following lemma will be helpful to compute the $K$-theory of the
$C^*$-algebra of the Penrose hyperbolic tiling. For short, we will
denote from now on  $\K({L}^2([0,1]))$ by  $\K$.
\begin{lemma}\label{lem-action-k1}
  Let ${A}$ be the unitarisation of $\K\otimes C(\Omega)\rtimes\Z$ and
let $f$ be a norm one function of ${L}^2([0,1])$. Then the
unitaries $$(1-\Theta_{f,f}\otimes 1)+\Theta_{f,f}\otimes u$$
 and  \begin{equation}\label{equ-action-k1}
\Theta_{U_0f,U_1f}\otimes u+\Theta_{U_1f,U_0f}\otimes
1+1-\Theta_{U_0f,U_0f}\otimes 1-\Theta_{U_1f,U_1f}\otimes 1
\end{equation}
of $A$  are homotopic.
\end{lemma}
\begin{proof}
If we set $f_0=f$ and complete to a  Hilbertian base
$f_0,\ldots,f_n,\ldots$ of ${L}^2([0,1])$, then
$U_0f_0,\ldots,U_0f_n,\ldots;U_1f_0,\ldots,U_1f_n,\ldots$
is a  Hilbertian basis of ${L}^2([0,1])$. In this base the unitary of
equation (\ref{equ-action-k1}) can be written down as
$$\left(\begin{array}{c|c}\begin{matrix} 0&&\\&1&\\&&\ddots \end{matrix}
&\begin{matrix} u&&\\&0&\\&&\ddots \end{matrix}\\
\hline
\begin{matrix} 1&&\\&0&\\&&\ddots \end{matrix}&\begin{matrix}
  0&&\\&1&\\&&\ddots \end{matrix}
\end{array}\right)$$ which is homotopic to $$\left(\begin{array}{c|c}\begin{matrix} u&&\\&1&\\&&\ddots \end{matrix}
&\begin{matrix} 0&&\\&0&\\&&\ddots \end{matrix}\\
\hline
\begin{matrix} 0&&\\&0&\\&&\ddots \end{matrix}&\begin{matrix}
  1&&\\&1&\\&&\ddots \end{matrix}
\end{array}\right).$$ All unitaries that can be writen down in such way in
some hilbertian basis of  ${L}^2([0,1])$ are homotopic and since this is
the case for
$1-\Theta_{f,f}\otimes 1+\Theta_{f,f}\otimes u$, we get the result.
\end{proof}

\section{The $C^*$-algebra of a Penrose hyperbolic tiling}\label{sec-c*-tiling}
Let us consider the semi-direct product groupoid
$\G=(X^\cN_{\mathcal{P}}\times Z_w)\rtimes\R$ corresponding to the
diagonal action of $\R$ on $X^\cN_{\mathcal{P}}\times Z_w$, by
translations on $X^\cN_{\mathcal{P}}$ and  trivial
on $Z_\omega$. 
Let us denote by $\lambda=(\lambda_{(\P',\omega)})_{(\P',\omega)\in
  X^\cN_{\mathcal{P}}\times Z_w}$ the Haar system provided by the left
  Haar mesure on $\R$. 
Let us define the groupoid automorphism
{${\alpha_w}:\G\to\G;\,(\P',w',t)\mapsto(R\cdot\P',\sigma(w'),2t)$}.
Then ${\alpha_w}$ preserves the Haar system $\lambda$ with  constant
density $\rho_{\alpha_w}=1/2$ and thus according to lemma \ref{rem-pr}
the suspension groupoid $\G_{\alpha_w}$ admits a Haar system
$\lambda_{\alpha_w}$. 
The semi-direct product groupoid $X^G_{\mathcal{P}(w)}\rtimes\R$,
where $\R$ acts on $X^G_{\mathcal{P}(w)}$ by translations, is equipped
with an action of $\R$ by automorphisms 
$\beta_t:X^G_{\mathcal{P}(w)}\rtimes\R\to X^G_{\mathcal{P}(w)}\rtimes\R;\,(\mathcal{\T},s)\mapsto
(2^{t}\cdot \mathcal{\T},2^{t}s)$ for any  $t$ in $\R$. The
automorphism $\beta_t$ preserves the Haar system with constant density
$2^{-t}$
and thus in view  of proposition \ref{prop-iso} induced a strongly
continuous action of $\R$ on the crossed product $C^*$-algebra
$C(X^G_{\mathcal{P}(w)})\rtimes\R$. 
\begin{lemma}\label{lem-psi}
Let $w$ be an element of $\{1,\ldots,r\}^\Z$. Then there is a unique
isomorphism of groupoids
$\Phi_w:\G_{\alpha_w}\longrightarrow
X^G_{\mathcal{P}(w)}\rtimes\R$ such that:
\begin{enumerate}
\item
  $\Phi_w([\mathcal{P}+x,w,y,0])=(\mathcal{P}(w)+x,y)$ for all $x$ and $y$ in $\R$;
\item $\Phi_w$ is equivariant with respect to the actions of
  $\R$;
\item $\Phi_{w,*}\lambda_{\alpha_w}$  is the Haar system on
  $X^G_{\mathcal{P}(w)}\rtimes\R$ provided by the Haar measure on
  $\R$.
\end{enumerate}
\end{lemma}
\begin{proof} With notations of the proof of proposition
  \ref{caracterisation}, let us define
  $\mathcal{T}(w')=\overline{\Phi}([T,w,1])$, were 
 $\mathcal{T}$ is in  $X^\cN_{\mathcal{P}}$  and $w'$ is 
in $\{1,...,r\}^\Z$.  Then the map  $$X^\cN_{\mathcal{P}}\times
Z_w\to X^G_{\mathcal{P}(w)};\,(\mathcal{T},w')\mapsto\mathcal{T}(w')$$
is continuous and  since
$(R \cdot \mathcal{T})(\sigma(w'))= R \cdot \mathcal{T}(w')$,
 the continuous map 
$$\G\times\R\to
X^G_{\mathcal{P}(w)}\rtimes\R;\,
(\mathcal{T},w',x,y)\mapsto (R_{2^{y}}\mathcal{T}(w'), 2^{y}x)$$
induces a  continuous homomorphism  of groupoids 
$$\Phi_w:\G_{\alpha_w}\to
X^G_{\mathcal{P}(w)}\rtimes\R.$$
This map is clearly one-to-one  since the equality
{$R_{2^t}\mathcal{T}(w')=\mathcal{T'}(w'')$} for $t$ in $\R$,
$\mathcal{T}$ and $\mathcal{T'}$  in $X^\cN_{\mathcal{P}}$ and $w''$ and $w'$ in
$Z_w$ holds if and only if $t$ is integer, 
$w''=\sigma^t(w')$ and {$R_{2^t} \mathcal{T}=\mathcal{T'}$}. To
prove that $\Phi_w$ is onto, let us remark that any element of 
$X^G_{\mathcal{P}(w)}$ can be written as {$R_{2^a}\mathcal{T}(w')$}, with
$a$ in $\R$, $\mathcal{T}$ in
$X^\cN_{\mathcal{P}}$ and $w'$ in $Z_\omega$. We get then  $$\Phi_w([\mathcal{T},w'
,2^{-a}t,a])=(R_{2^{a}}\mathcal{T}(w'),t)$$ for all $t$ in $\R$.

It is then straightforward to check that  condition (3) of the
lemma is satisfied. The uniqueness of $\Phi_w$ is a consequence
 on one hand of its equivariance  and on the other hand of the density
 of the $\R$-orbit 
of $\P$ in $X^\cN_{\mathcal{P}}$.
\end{proof}
As a consequence of lemma \ref{lem-psi}, we get
\begin{corollary}
The map $$C_c( X^G_{\mathcal{P}(w)}\rtimes\R)\to C_c(\G_{\alpha_w});\,f\mapsto f\circ \Phi_w$$ induces
  an $\R$-equivariant isomorphism
  $$\widetilde{\Phi}_w:C_0(X^G_{\mathcal{P}(w)})\rtimes\R
\to C^*_r(\G_{\alpha_w},\lambda_{\alpha_w}).$$
\end{corollary}
\begin{proposition}\label{prop-morita}
Using the notations of lemmas  \ref{lem-mt} and \ref {lem-psi}, the $C^*$-algebras
$C(X^G_{\mathcal{P}(w)})\rtimes G$ and
$C_r^*(\G,\lambda)\rtimes_{\tilde{\alpha_w}}\Z$   are Morita equivalent.
\end{proposition}
\begin{proof}
Recall that $G=\R\rtimes\R^*_+$, where the group $(\R^*_+,\cdot)$ acts on
$(\R,+)$ by multiplication. Iterate   crossed products  leads to
an isomorphism $$C(X^G_{\mathcal{P}(w)})\rtimes G\cong
(C(X^G_{\mathcal{P}(w)})\rtimes\R)\rtimes\R^*_+.$$If we identify the
groups $(\R,+)$ and $(\R^*_+,\cdot)$ using the isomorphism
$$\R\to\R^*_+;\,t\mapsto 2^{t},$$ this provides  the action  under consideration
in lemma \ref{lem-psi} of $\R$ on 
$C(X^G_{\mathcal{P}(w)})\rtimes\R$
and 
hence, the algebras 
$C(X^G_{\mathcal{P}(w)})\rtimes G$ and
$C^*(\G_{\alpha_w},\lambda_{\alpha_w})\rtimes\R$ are isomorphic. In
view of lemma \ref{prop-iso} and of remark \ref{rem-action}, the $C^*$-algebra
$C(X^G_{\mathcal{P}(w)})\rtimes G$ is isomorphic to
$C_r^*(\G,\lambda)_{\tilde{\alpha_w}}\rtimes\R$. But since
$C_r^*(\G,\lambda)_{\tilde{\alpha_w}}$ is the mapping torus algebra with
respect to the automorphism $\tilde{\alpha_w}:C_r^*(\G,\lambda)\to C_r^*(\G,\lambda)$, the crossed product $C^*$-algebra
$C_r^*(\G,\lambda)_{\tilde{\alpha_w}}\rtimes\R$ is Morita equivalent to
$C_r^*(\G,\lambda)\rtimes_{\tilde{\alpha_w}}\Z$  and hence we get
the result.
\end{proof}

\section{The $K$-theory of the $C^*$-algebra of a Penrose hyperbolic
  tiling} \label{section_K_th}

Let us consider the semi-direct groupoid $\G=(X^\cN_{\mathcal{P}}\times Z_w)\rtimes\R$ corresponding to the
diagonal action of $\R$ on $X^\cN_{\mathcal{P}}\times Z_w$, by
translations on $X^\cN_{\mathcal{P}}$ and  trivial
on $Z_w$. 
According to proposition \ref{prop-morita} we have an isomorphism
$$K_*(C(X^G_{\mathcal{P}(w)})\rtimes G)\stackrel{\cong}{\to}K_*(
C_r^*(\G,\lambda)\rtimes_{\tilde{\alpha_w}}\Z)$$ induced by the
Morita equivalence. In order to compute this $K$-theory group,  we
need to recall some basic facts concerning the $K$-theory group of a
crossed product of a $C^*$-algebra $A$ by an action of $\Z$ provided
by an automorphism $\theta$ of $A$. This $K$-theory can be computed
by using the Pimsner-Voiculescu exact sequence \cite{pv}
$$\begin{CD}
K_0(A)@>\theta_*-Id>>K_0(A)@>\iota_*>>K_0(A\rtimes_\theta\Z)\\
@AAA @.  @VVV\\
K_1(A\rtimes_\theta\Z)@<\iota_*<<K_1(A)@<\theta_*-Id<<K_1(A)
\end{CD},$$ where $\iota_*$ is the homomorphism induced in $K$-theory by
the inclusion $\iota:A\hookrightarrow A\rtimes_\theta\Z$ and $\theta_*$
is the homomorphism in $K$-theory induced by $\theta$.
The vertical maps  are given by
the composition 
$$K_*(A\rtimes_\theta\Z)\stackrel{\cong}{\lto} K_*(A_\theta\rtimes_{\widehat{\theta}}\R)\stackrel{\cong}{\lto} 
K_{*+1}(A_\theta)
\stackrel{{ev}_*}{\lto}K_{*+1}(A),$$
where
\begin{itemize}
\item $A_\theta$ is the mapping torus of $A$ with respect to the
  action $\theta$ endowed, with its  associated action
  $\widehat{\theta}$ of
  $\R$;
\item the first map is induced by the Morita equivalence between
  $A\rtimes_\theta\Z$ and $A_\theta\rtimes_{\widehat{\theta}}\R$;
\item the second map is the Thom-Connes isomorphism;
\item the third map is induced in K-theory by the evaluation map
$${ev}:A_\theta\to A;\, f\mapsto f(0).$$
\end{itemize}
  For an automorphism $\Psi$ of an abelian group $M$, let us define  
 $\inv M$  as the set of invariant elements of $M$ and by $\co M=M/\{x-\Psi(x),x\in M\}$
 the set of coinvariant  elements. 
We then get   short exact sequences
\begin{equation}\label{equ-pv0}
0\to \co K_0(A)\to K_0(A\rtimes_\theta\Z) \to \inv K_1(A)\to 0
\end{equation}
and 
\begin{equation}\label{equ-pv1}
0\to \co K_1(A)\to K_1(A\rtimes_\theta\Z) \to \inv K_0(A)\to 0.
\end{equation}
 Moreover the inclusions in these exact
sequences are induced by $\iota_*$.
The first step in the computation of  $K_*(
C_r^*(\G,\lambda)\rtimes_{\tilde{\alpha_w}}\Z)$ is provided by
next  lemma, which is straightforward to prove.
\begin{lemma}\label{lem-cantor}
Let $Z$ be a Cantor set and let us denote by $C(Z,\Z)$ the algebra of
continuous and integer valued functions on $Z$. 
\begin{enumerate}
\item we have an isomorphism $C(Z,\Z)\to K_0(C(Z));\,\chi_E\mapsto
  [\chi_E]$.
\item $K_1(C(Z))=\{0\}$,
\end{enumerate}
where for a  compact-open
subset $E$ of $Z$,  then $\chi_E$ stands for the characteristic function of
$E$.
\end{lemma}
Plugging $C_r^*(\G,\lambda)\rtimes_{\tilde{\alpha_w}}\Z$ into the
short exact sequences (\ref{equ-pv0}) and  (\ref{equ-pv1}), we get 
\begin{equation}\label{eq-exsq1}
0\to \co K_0(C_r^*(\G,\lambda))\to K_0(C_r^*(\G,\lambda)\rtimes_{\tilde{\alpha_w}}\Z) \to \inv K_1(C_r^*(\G,\lambda))\to 0
\end{equation}
and 
\begin{equation}\label{eq-exsq2}
0\to \co K_1(C_r^*(\G,\lambda))\to K_1(C_r^*(\G,\lambda)\rtimes_{\tilde{\alpha_w}}\Z) \to \inv K_0(C_r^*(\G,\lambda))\to 0.
\end{equation}
According to equation  (\ref{equ-Morita}), the $C^*$-algebra
$C_r^*(\G,\lambda)$  is isomorphic to $C(Z_w)\otimes\K\otimes
C(\Omega)\rtimes\Z$. The $K$-theory of $C_r^*(\G,\lambda)$ can be
the computed by using the  K\"unneth formula: in view of lemma
\ref{lem-cantor}, $K_0(C(Z_w))\cong C(Z_w,\Z)$ is torsion
free and $K_1(C(Z_w))=\{0\}$ and by Morita equivalence, we get that
$$K_0(C_r^*(\G,\lambda)) \cong C(Z_w,\Z)\otimes
K_0(C(\Omega)\rtimes\Z)$$ and
$$K_1(C_r^*(\G,\lambda)) \cong C(Z_w,\Z)\otimes
K_1(C(\Omega)\rtimes\Z).$$ This isomorphism, up to the Morita
equivalence and to the isomorphism of equation (\ref{equ-Morita}) are implemented
by the external product in $K$-theory and will be precisely described
later on. Once again, $K_*(C(\Omega)\rtimes\Z)$ can be computed from
the short exact sequences (\ref{equ-pv0}) and (\ref{equ-pv1}), and we get, using
lemma \ref{lem-cantor} that
\begin{equation}\label{equ-mt1}
K_0(C(\Omega)\rtimes\Z)\cong\co C(\Omega,\Z)
\end{equation}
and
\begin{equation}\label{equ-mt2}
K_1(C(\Omega)\rtimes\Z)\cong\inv C(\Omega,\Z)\cong\Z.
\end{equation}
The isomorphism of equation (\ref{equ-mt1}) is induced by the composition
$$C(\Omega,Z)\stackrel{\cong}{\to}K_0(C(\Omega))\to
K_0(C(\Omega)\rtimes\Z),$$
which factorizes through $\co
C(\Omega,Z)$, where the first map is described in lemma
\ref{lem-cantor}, and the second map is induced on $K$-theory by the
inclusion $C(\Omega)\hookrightarrow C(\Omega)\rtimes\Z$.
In the  first isomorphism of equation (\ref{equ-mt2}) the class of $[u]$
in $K_1(C(\Omega)\rtimes\Z)$ of the unitary $u$ of
$C(\Omega)\rtimes\Z$ corresponding to the positive generator of $\Z$
is mapped to the constant function $1$ of $C(\Omega,\Z)$.
\begin{lemma}\label{lem-coinv}
Let $\nu$ be the Haar measure on $\Omega$. Then
\begin{enumerate}
\item $\int fd\nu$ is in $\Z[1/2]$ for all $f$ in $C(\Omega,\Z)$;
\item $C(\Omega,\Z)\to \Z[1/2];f\mapsto\int fd\nu$ factorizes through
  an isomorphism 
$$\co C(\Omega,\Z)\stackrel{\cong}{\to}\Z[1/2].$$
\end{enumerate}
\end{lemma}
\begin{proof}It is enought to check the first point for characteristic
  function of compact-open subset of $\Omega$.
For  an integer $n$ and $k$ in $\{0,\ldots,2^{n-1}\}$, we set
$F_{n,k}=2^n\Omega+k$. Then $(F_{n,k})_{n\in\N,\,0\leq k\leq 2^{n-1}}$
is a basis of compact-open neighborhoods for $\Omega$ and thereby,
every compact-open subset of $\Omega$ is a finite disjoint union of some
$F_{n,k}$. Since $\nu(F_{n,k})=2^{-n}$, we get the first point.

The measure $\mu$ being invariant by translation, the map
 $$C(\Omega,\Z)\to \Z[1/2];f\mapsto\int fd\nu$$ factorizes through a
 group 
 homomorphism $\co C(\Omega,\Z)\to \Z[1/2]$. This homomorphism admits  a
 cross-section
\begin{equation}\label{equ-cross-sec}
\Z[1/2]\to \co C(\Omega,\Z);\, 2^{-n}\mapsto [\chi_{F_{n,0}}].
\end{equation}
This map is well defined since $F_{n,0}=F_{n+1,0}\coprod (2^n+F_{n+1,0})$
and thus
$$[\chi_{F_{n,0}}]=[\chi_{F_{n+1,0}}]+[\chi_{2^n+F_{n+1,0}}]=2[\chi_{F_{n+1,0}}]$$
in $\co C(\Omega,\Z)$. Since the
 $(\chi_{F_{n,k}})_{n\in\N,\,0\leq k\leq 2^{n-1}}$ generates
 $C(\Omega,\Z)$ as an abelian  group, it is enought to check that the
 cross-section of equation (\ref{equ-cross-sec}) is a left
inverse on  $\chi_{F_{n,k}}$, which is true since $[\chi_{F_{n,k}}]=
[\chi_{k+F_{n,0}}]=[\chi_{F_{n,0}}]$ in $\co C(\Omega,\Z)$.
\end{proof}
\begin{proposition}\label{prop-Rhull}
Let $C(Z_w,\Z[1/2])\cong C(Z_w,\Z)\otimes\Z[1/2]$ be the
algebra of continuous function on $Z_w$, valued in $\Z[1/2]$
(equipped with the discrete topology).
Then with the notations of the proof of lemma \ref{lem-coinv}, we have
isomorphisms
\begin{enumerate}
\item \begin{eqnarray*}
C(Z_w,\Z[1/2])&\stackrel{\cong}{\longrightarrow}&K_0(C(Z_w)\otimes
C(\Omega)\rtimes\Z)\\
\frac{\chi_E}{2^{n}}&\mapsto&[\chi_E\otimes\chi_{F_{n,0}}],
\end{eqnarray*}
where  $E$ is a compact-open subset
of
$Z_w$ and   $\chi_E$  is its  characteristic function.

\item \begin{eqnarray*}
C(Z_w,\Z)&\stackrel{\cong}{\longrightarrow}&K_1(C(Z_w)\otimes
C(\Omega)\rtimes\Z)\\
{\chi_E}&\mapsto&[\chi_E\otimes u+(1-\chi_E)\otimes 1],
\end{eqnarray*}where $u$ is the unitary of $C(\Omega)\rtimes\Z$
corresponding to the positive generator of $\Z$.
\end{enumerate}
\end{proposition}
\begin{proof} As we have already mentionned, $K_*(C(Z_w))$ is
  torsion free and the  K\"unneth formula
  provides   isomorphisms
\begin{eqnarray*}
K_0(C(Z_w))\otimes K_0(C(\Omega)\rtimes\Z)&\stackrel{\cong}{\to}&K_0(C(Z_w)\otimes
C(\Omega)\rtimes\Z)\\
\lbrack p\rbrack \otimes\lbrack q\rbrack&\mapsto &\lbrack p\otimes q\rbrack,
\end{eqnarray*}
where $p$ and $q$ are some matrix projectors   with  coefficients  respectively
in $C(Z_w)$ and  $C(\Omega)\rtimes\Z$, and 
 
\begin{eqnarray*}
K_0(C(Z_w))\otimes K_1(C(\Omega)\rtimes\Z)&\stackrel{\cong}{\to}&K_1(C(Z_w)\otimes
C(\Omega)\rtimes\Z)\\\lbrack p\rbrack\otimes\lbrack v\rbrack&\mapsto
&\lbrack p\otimes v+(I_k-p)\otimes I_l\rbrack,
\end{eqnarray*}
where $p$ is a projector in $M_l(C(Z_w))$ and $v$ is a unitary in
$M_k(C(\Omega)\rtimes\Z)$. The proposition is  then  a  consequence
of lemmas \ref{lem-cantor}, \ref{lem-coinv} and of the discussion
related to equations (\ref{equ-mt1}) and (\ref{equ-mt2}).
\end{proof}
In order to compute the invariants and the coinvariants of $$K_*(
C_r^*(\G,\lambda))\cong K_*(C(Z_w)\otimes\K\otimes
C(\Omega)\rtimes\Z),$$ we will need a carefull description of  the
action induced in $K$-theory by the automorphism $\sigma^*\otimes\Upsilon$
of $C(Z_w)\otimes\K\otimes
C(\Omega)\rtimes\Z$, where $\sigma^*$ is the automorphism of
$C(Z_\omega)$ induced by the shift $\sigma$ and where $\Upsilon$ was defined in
section \ref{sec-penrose-uncolor}.
\begin{lemma}
\label{lem-mt1}
If we equip $C(Z_w)\otimes\K\otimes
C(\Omega)\rtimes\Z$ with the $\Z$-action provided by
$\sigma^*\otimes\Upsilon$ and under the $\Z$-equivariant isomorphism
$$C_r^*(\G,\lambda)\cong C(Z_w)\otimes\K\otimes
C(\Omega)\rtimes\Z,$$ the action induced by $\alpha_w$ on $K_0(
C_r^*(\G,\lambda)\cong C(Z_w,\Z[1/2])$ and on  $K_1(
C_r^*(\G,\lambda)\cong C(Z_w,\Z)$ are given by the automorphisms
of abelian groups
\begin{eqnarray*}
\Psi_0: C(Z_w,\Z[1/2])&\to& C(Z_w,\Z[1/2])\\
f&\mapsto& 2f\circ \sigma^{-1}
\end{eqnarray*}
and 
\begin{eqnarray*}
\Psi_1: C(Z_w,\Z)&\to& C(Z_w,\Z)\\
f&\mapsto& f\circ \sigma^{-1}.
\end{eqnarray*}
\end{lemma}
\begin{proof}
According to proposition \ref{prop-Rhull}
and using the Morita equivalence  between $C(Z_w)\otimes
C(\Omega)\rtimes\Z$  and  $C(Z_w)\otimes\K\otimes
C(\Omega)\rtimes\Z$, in order to describe  $\Psi_0$, we have to compute the
image under $(\sigma^*\otimes\Upsilon)_*$ of
$$[\chi_E\otimes\Theta_{f,f}\otimes\chi_{F_{n,0}}]\in K_0(C(Z_w)\otimes\K\otimes
C(\Omega)\rtimes\Z)$$ where,
\begin{itemize}
\item $\chi_E$ is the characteristic function of a compact-open subset
  $E$ of $Z_w$;
\item $\chi_{F_{n,0}}$ is the characteristic function of
  $F_{n,0}=2^n\Omega$ for $n\geq 1$;
\item $\Theta_{f,f}$ is the rank one projector associated to a norm $1$
  function
  $f$ of $L^2([0,1])$.
\end{itemize}
We have 
\begin{eqnarray*}
\sigma^*\otimes\Upsilon(\chi_E\otimes\Theta_{f,f}\otimes\chi_{F_{n,0}})&=&\chi_{\sigma(E)}\otimes\Theta_{U_0f,U_0f}\otimes
W_0\chi_{F_{n,0}}+\chi_{\sigma(E)}\otimes\Theta_{U_1f,U_1f}\otimes W_1\chi_{F_{n,0}}\\
&=&\chi_{\sigma(E)}\otimes\Theta_{U_0f,U_0f}\otimes\chi_{F_{n-1,0}},
\end{eqnarray*}where the last equality holds since
$W_0\chi_{F_{n,0}}=\chi_{F_{n-1,0}}$ and $W_1\chi_{F_{n,0}}=0$. Since
      $\Theta_{U_0f,U_0f}$ is again a rank one projector, then up
to the  Morita equivalence  between $C(Z_w)\otimes
C(\Omega)\rtimes\Z$  and  $C(Z_w)\otimes\K\otimes
C(\Omega)\rtimes\Z$, the image of $[\chi_E\otimes\chi_{F_{n,0}}]\in K_0(C(Z_w)\otimes
C(\Omega)\rtimes\Z)$ under $(\sigma^*\otimes\Upsilon)_*$ is $[\chi_{\sigma(E)}\otimes\chi_{F_{n-1,0}}]\in K_0(C(Z_w)\otimes
C(\Omega)\rtimes\Z)$. Using proposition \ref{prop-Rhull}, this 
completes the  description of $\Psi_0$. For $\Psi_1$, notice first that up
to the isomorphism
$$K_0(C(Z_w))\otimes K_0(\K\otimes
C(\Omega)\rtimes\Z)\stackrel{\cong}{\to} K_1(C(Z_w)\otimes\K\otimes
C(\Omega)\rtimes\Z)$$ provided by the K\"unneth formula, the action of
$(\sigma^*\otimes\Upsilon)_*$ is $\sigma^*_*\otimes\Upsilon_*$ and then the
result is a consequence of lemma \ref{lem-action-k1} and of
proposition \ref{prop-Rhull}.
\end{proof}
Let us equip $C(Z_w,\Z[1/2])$ and  $C(Z_w,\Z)$ with the
$\Z$-actions respectively provided by $\Psi_0$ and $\Psi_1$. Then
since $\|\Psi_0(h)\|=2\|h\|$ for any $h$ in $C(Z_w,\Z[1/2])$, we
get that $\inv C(Z_w,\Z[1/2])=\{0\}$
We are now in position to get a complete description of the $K$-theory
of $C(X^G_{\mathcal{P}(w)})\rtimes G$. In view of the short exact
sequences of equations (\ref{eq-exsq1}) and (\ref{eq-exsq2}), the two following theorems are
then consequences of lemma \ref{lem-mt1} and of proposition \ref{prop-Rhull}.
\begin{theorem}\label{thm-k0-gl}
We have a short exact sequence
$$0\to \co C(Z_w,\Z[1/2])\stackrel{\iota_0}{\to}
K_0(C(X^G_{\mathcal{P}(w)})\rtimes G)\to\inv C(Z_w,\Z)\to 0,$$
where
up to the Morita equivalence $C(X^G_{\mathcal{P}(w)})\rtimes G\cong C_r^*(\G,\lambda)\rtimes_{\tilde{\alpha_w}}\Z$, the element
$\iota_0[2^{-n}\chi_E]$ is the image of
$[\chi_E\otimes\Theta_{f,f}\otimes\chi_{F_{n,0}}]\in K_0(C(Z_w)\otimes\K\otimes
C(\Omega)\rtimes\Z)$ under the homomorphism induced in $K$-theory by the
inclusion
$$C(Z_w)\otimes\K\otimes
C(\Omega)\rtimes\Z\cong C_r^*(\G,\lambda)\hookrightarrow
C_r^*(\G,\lambda)\rtimes_{\tilde{\alpha_w}}\Z,$$
where\begin{itemize}
\item $\chi_E$ is the characteristic function of a compact-open subset
  $E$ of $Z_w$;
\item $\chi_{F_{n,0}}$ is the characteristic function of
  $F_{n,0}=2^n\Omega$;
\item $\Theta_{f,f}$ is the rank one projector associated to a norm $1$
  function
  $f$ of $L^2([0,1])$.
\end{itemize}
\end{theorem}

\begin{theorem}We have an isomorphism $$\co
  C(Z_w,\Z)\stackrel{\cong}{\to}K_1(C(X^G_{\mathcal{P}(w)})\rtimes
  G)$$ induced on the coinvariants by the composition
$$ C(Z_w,\Z)\cong K_0(C(Z_w))\stackrel{\otimes[u]}{\to}
K_1(C(Z_w)\otimes C(\Omega)\rtimes\Z)\cong
K_1(C_r^*(\G,\lambda)\to
K_1(C_r^*(\G,\lambda)\rtimes_{\tilde{\alpha_w}}\Z),$$where
\begin{itemize}
\item $\otimes[u]$ is the external product in $K$-theory by the class
  in $K_1(C(\Omega)\rtimes\Z)$ of the unitary $u$ of
  $C(\Omega)\rtimes\Z$   corresponding to the positive generator of
  $\Z$;
\item the last map  in the composition  is the
  homomorphism induced in $K$-theory by the inclusion
  $C_r^*(\G,\lambda)\hookrightarrow
  C_r^*(\G,\lambda)\rtimes_{\tilde{\alpha_w}}\Z$.
\end{itemize}
\end{theorem}
The short exact sequence of theorem \ref{thm-k0-gl}, admits an
explicit 
splitting  which can be described in the following way:
assume first that $(Z_w,\sigma)$ is minimal. In particular,
$\inv C(Z_w,\Z)\cong\Z$ is generated by $1\in C(Z_w,\Z)$.
Let us considerer the following diagram, whose left square is commutative
$$\begin{CD}
K_1(C^*(\R))@>>>K_1(C^*(\G,\lambda))\\
@VVV   @VVV\\
\Z\cong K_0(\C)@>>>K_0(C(X^\cN_{\mathcal{P}}\times Z_w))@>{ev}_*>>K_0(C(\Omega\times Z_w))\\
\end{CD},$$ where 
\begin{itemize}
\item the horizontal maps of the left square are induced by the inclusion
 $\C\hookrightarrow C(X^\cN_{\mathcal{P}}\times
Z_w)$.
\item vertical maps are  the Thom-Connes {isomorphisms}.
\item The map ${ev}:C(X^\cN_{\mathcal{P}}\times Z_w))\lto
  C(\Omega\times Z_w)$ is induced by the continuous map
  $\Omega\to X^\cN_{\mathcal{P}}\cong (\Omega\times \R)/\cA_o;\,
  x\mapsto [x,0]$;
\end{itemize}
Up to the Morita equivalence between $C^*(\G,\lambda)$
and
  $C(Z_w)\ts C(\Omega)\rtimes\Z$, the right down staircase
  is the boundary of the Pimsner-Voiculescu six-term exact
  sequence that computes $K_*(C(Z_w)\ts C(\Omega)\rtimes\Z)$. From this, we see that 
$K_1(C^*(\G,\lambda)\cong\Z$ is generated by the image  of the
 generator $\zeta$ of $K_1(C^*(\R))$ corresponding under the canonical
 identification $K_1(C^*(\R))\cong  K_1(C_0(\R))\cong  K_0(\C)\cong
 \Z$ to the class of any rank one projector in some $M_n(\C)$.
On the other hand, we have a  diagram with commutative squares
{\small{$$\begin{CD}
K_0(C^*(\R)\rtimes\R_+^*)@>>>K_0(C(X^G_{\mathcal{P}(w)})\rtimes
G)@>>\cong>K_0(C^*(\G_{\alpha_w},\lambda_{\alpha_w})\rtimes \R)@>>\cong>K_0(C^*(\G,\lambda)_{\alpha_w}\rtimes\R)\\
@VVV   @VVV @VVV @VVV\\
K_1(C^*(\R))@>>>K_1(C(X^G_{\mathcal{P}(w)})\rtimes
\R)@>>\cong>K_1(C^*(\G_{\alpha_w},\lambda_{\alpha_w}))@>>\cong>K_1(C^*(\G,\lambda)_{\alpha_w})\\
@.@.@.@VV{ev}_*V\\
&&&&&&K_0(C^*(\G,\lambda))
\end{CD}$$ }}
where,
\begin{itemize}
\item the horizontal maps of the left square are induced by the
  inclusion\\
$C^*(\R)\hookrightarrow C(X^G_{\mathcal{P}(w)})\rtimes
\R;$
\item the horizontal maps of the middle  square are induced by the
  isomorphism of lemma \ref{lem-psi}
\item the horizontal maps of the right   square are induced by the
  isomorphism of proposition \ref{prop-iso}
\item the first row of vertical maps are Thom-Connes isomorphisms.
\end{itemize}
It is then straightforward to check that the down staircase of the
diagram is indeed induced by the inclusion $C^*(\R)\hookrightarrow C(X^\cN_{\mathcal{P}}\times
Z_w)\rtimes\R=C^*(\G,\lambda)$.  
Notice that  $K_0(C^*(\R)\rtimes\R_+^*)\cong\Z$ (by  Thom-Connes
isomorphism). Moreover, under the inclusion
$C_0(\R_+^*)\rtimes\R_+^*\hookrightarrow
C_0(\R)\rtimes\R_+^*\cong C^*(\R)\rtimes\R_+^*$, any rang one
projector $e$ of  $\K(L^2(\R_+^*))\cong C_0(\R_+^*)\rtimes\R_+^*$ provides a generator
for  $K_0(C^*(\R)\rtimes\R_+^*)$ whose image under the left vertical
map is the generator $\zeta$ for $K_1(C^*(\R))\cong K_0(\C)\cong
\Z$. Using the description of the boundary map of
Pimsner-Voiculescu six-term exact sequence, we see that  $e$,  
viewed as an element   of  $C(X^G_{\mathcal{P}(w)})\rtimes
G$ whose class in $K$-theory provides a lift  for $1\in C(Z_w,\Z)$ in the short exact
sequence of theorem \ref{thm-k0-gl}.

\medskip

In general, $\inv C(Z_w,\Z)$ is generated by
characteristic functions  of $\Z$-invariant compact-open subsets of
$Z_w$.
According to proposition \ref{caracterisation}, any $\Z$-invariant
compact-open subset $E$  of
$Z_w$ provides a $\R$-invariant compact subset $\widetilde{E}$ of
$X^G_{\mathcal{P}}$. Hence, with above notations, if
  $\chi_{\widetilde{E}}$ is the characteristic function for
  $\widetilde{E}$, then  $\chi_{\widetilde{E}}e$ can be viewed as an
    element of $C(X^G_{\mathcal{P}(w)})\rtimes
G$. Let 
$\upsilon:\inv C(Z_w,\Z)\to K_0(C(X^G_{\mathcal{P}(w)})\rtimes
G)$ be the group homomorphism uniquelly defined by   $\upsilon(\chi_E)=\chi_{\widetilde{E}}e$
 for $E$ a  $\Z$-invariant compact-open subset of
$Z_w$. Then $\upsilon$ is  a section for the short exact
sequence of theorem \ref{thm-k0-gl}.

\section{Topological invariants for the continuous hull}

It is known that for Euclidian tilings, topological
invariants of the continuous hull are closely related to the
$K$-theory of the  $C^*$-algebra associated to the tiling. The
$K$-theory of the latter {turn out to be} isomorphic to the $K$-theory of the
hull which is using the Chern character rationally isomorphic to the
integral C\v ech cohomology. Moreover, in dimension less or equal
to $3$, the Chern character can be defined valued in integral
cohomology and we eventually obtain an  isomorphism between the integral  C\v ech
cohomology of the hull and the
$K$-theory of the  $C^*$-algebra associated to the tiling. In
consequence of this fact, a lot of interest has been generated in the
computation of topological invariants of the hull.

For Penrose hyperbolic tilings, since the group of affine isometries of the hyperbolic half-plane is isomorphic to a
semi-direct product  $\R\rtimes\R$, we get using the Thom-Connes isomorphism that  
\begin{equation}\label{equ-TC}
K_*(C(X^G_{\mathcal{P}(w)})\rtimes G)\cong
K^*(X^G_{\mathcal{P}(w)}).\end{equation} Moreover, since the cohomological
dimension of  $X^G_{\mathcal{P}(w)}$ is $2$, the Chern character can
also be defined with values in integral C\v ech cohomology and hence we
get as in the Euclidian case of low dimension
 an isomorphism $$K_*(C(X^G_{\mathcal{P}(w)})\rtimes G)\cong
\check{H}(X^G_{\mathcal{P}(w)},\Z).$$ These topological invariants can be
  indeed computed directly using technics very closed to those used in
  section \ref{section_K_th}. Indeed for a $C^*$-algebra $A$ provided
  with an  automorphism $\beta$, there is natural isomorphisms
 \begin{equation}\label{equ-mt-iso}K_0(A_\beta)\cong
    K_1(A\rtimes_\beta\Z)\quad\text{ and }\quad K_1(A_\beta)\cong
    K_0(A\rtimes_\beta\Z),\end{equation}
called  the  mapping torus isomorphisms, where $A_\beta$ is the mapping torus algebra constructed at the end of section
\ref{subsec-suspension}. Recall from proposition \ref{caracterisation} that
   $X^G_{\mathcal{P}(w)}$ can be viewed as a double suspension
\begin{equation}\label{equ-double-suspension} \left(((\Omega \times \R) \slash \cA_{o})\times Z_\omega\times
   \R\right) \slash{\cA_{f}}\end{equation}  with 
$f:(\Omega \times \R) \slash \cA_{o} \to (\Omega \times \R) \slash
\cA_{o};\, ([x,t],\omega')\mapsto ([2x,2t],\sigma(\omega'))$. In regard
of the mapping torus isomorphism,  the double
  crossed product by $\Z$ corresponds in $K$-theory  to the double
  suspension structure on $X^G_{\mathcal{P}(w)}$. For people
  interested in  topological invariants, we explain how a straight
  computation can be carried out.

For a $C^*$-algebra $A$ provided with an automorphism $\beta$, the short exact sequence
\begin{equation}
\label{equ-mt}
0\to C_0((0,1),A)\to A_\beta\stackrel{ev}{\to} A\to  0,
\end{equation}
where $ev$ is the evaluation at $0$ of elements of $A_\beta\subset
C([0,1],A)$, gives rise to  short exact sequences
$$0\to \co K_1(A)\to K_0(A_\beta)\to \inv K_0(A)\to 0$$ and $$0\to \co
K_0(A)\to K_1(A_\beta)\to \inv K_1(A)\to 0$$ where invariants and
coinvariants are taken with respect to the action induced by $\beta$ on
$K_*(A)$ (see section \ref{section_K_th}). In particular, if $X$ is a compact set  and $f:X\to X$ is
a homeomorphism, and {with the} notations of section \ref{subsec-hull}
 the mapping torus of $C(X)$ with respect to the automorphism
induced by $f$ is  $C((X \times \R) \slash \cA_{f})$.  We deduce  short exact sequences
 $$0\to \co K^1(X)\to K^0((X \times \R) \slash \cA_{f})\to \inv K^0(X)\to 0$$ and $$0\to \co
K^0(X)\to K^1((X \times \R) \slash \cA_{f})\to \inv K^1(X)\to
0.$$ Similarly, we have short exact sequences in C\v ech cohomology
$$0\to \co \check{H}^{n-1}(X,\Z)\to \check{H}^n((X \times \R) \slash
\cA_{f},\Z)\to \inv \check{H}^n(X,\Z)\to 0,$$
derived from the inclusion \begin{equation}\label{equ-inclusion}
(0,1)\times X\hookrightarrow (X \times \R) \slash
\cA_{f}.\end{equation}
Since {the space} $X^G_{\mathcal{P}(w)}$ has a structure of double
suspension, we see  following the same route as in section
\ref{section_K_th}    that $X^G_{\mathcal{P}(w)}$ only
has cohomology in degree $0,\,1$ and $2$
 and we get isomorphisms
\begin{eqnarray}\label{equ-k0-hull}
K^0(X^G_{\mathcal{P}(w)})&\cong& \inv C(Z_w,\Z)\oplus \co C(Z_w,\Z[1/2])\\
\label{equ-k1-hull} K^1(X^G_{\mathcal{P}(w)})&\cong& \co C(Z_w,\Z)\\
\label{equ-h0-hull} \check{H}^0(X^G_{\mathcal{P}(w)},\Z)&\cong& \inv C(Z_w,\Z)\\
\label{equ-h1-hull} \check{H}^1(X^G_{\mathcal{P}(w)},\Z)&\cong& \co C(Z_w,\Z)\\
\label{equ-h2-hull} \check{H}^2(X^G_{\mathcal{P}(w)},\Z)&\cong&  \co C(Z_w,\Z[1/2]).
\end{eqnarray}
Recall that invariants and  coinvariants of $C(Z_w,\Z)$ are taken with
respect to the automorphism  $$C(Z_w,\Z)\to C(Z_w,\Z);\,
f\mapsto f\circ\si^{-1},$$
and that coinvariants of $C(Z_w,\Z[1/2])$ are taken with
respect to the automorphism  $$C(Z_w,\Z[1/2])\to C(Z_w,\Z[1/2]);\,
f\mapsto 2f\circ\si^{-1}.$$
Let us describe explicitly these isomorphisms.
The identification of equation (\ref{equ-double-suspension}) yields to a
continuous map \begin{equation}\label{equ_proj}
X^G_{\mathcal{P}(w)}\to (Z_w\times\R)\slash
\cA_{\sigma},\end{equation}induced by the equivariant projection $((\Omega \times \R) \slash \cA_{o})\times Z_\omega\times
   \R\to Z_w\times\R$. Together with the inclusion
$Z_w\times(0,1)\hookrightarrow (Z_w\times\R)\slash
\cA_{\sigma}$, this gives rise to a homomorphism
$K^1(Z_w\times(0,1))\to  K^1(X^G_{\mathcal{P}(w)})$ inducing under
Bott peridodicity the isomorphism of equation (\ref{equ-k1-hull})
(recall from lemma \ref{lem-cantor} that $K^0(Z_w)=K_0(C(Z_w))\cong
C(Z_\omega,\Z)$). The identification  of equation (\ref{equ-h1-hull}) is
obtained in the same way by using the isomorphism
\begin{equation*}C(Z_\omega,\Z)\cong\check{H}^0(Z_\omega,\Z)\stackrel{\cong}{\to}
  \check{H}^1(Z_\omega\times(0,1),\Z)\end{equation*}provided by the cup
product by  the fundamental class of $\check{H}^1((0,1),\Z)$.  Recall
that $\inv C(Z_w,\Z)$ is generated by characteristic functions of
invariant compact-open  subsets  of $Z_\omega$. If  $E$ is   such a
subset, then $(E\times\R)\slash
{\cA}_{\sigma_{\slash E}}$ is a compact-open subset of $(Z_w\times\R)\slash
\cA_{\sigma}$ and is pulled-back under the map of  equation
(\ref{equ_proj}) to a compact-open subset 
$\widetilde{E}$ of  $X^G_{\mathcal{P}(w)}$. The isomorphism of equation
(\ref{equ-h0-hull}) identifies $\chi_E\in \inv C(\Z_\omega,\Z)$  with the class of
${\widetilde{E}}$ in   $\check{H}^0(X^G_{\mathcal{P}(w)},\Z)$ and
 the isomorphism of equation (\ref{equ-k0-hull}) identifies $\chi_E$ with the class
of  $\chi_{\widetilde{E}}$ in  $K^0(Z_w)=K_0(C(Z_w))$. Using twice the
inclusion of equation (\ref{equ-inclusion}) for the double suspension
structure of $X^G_{\mathcal{P}(w)}$, we obtain an inclusion 
$\Omega \times Z_\omega\times (0,1)^2\hookrightarrow X^G_{\mathcal{P}(w)}$
and hence by Bott  periodicity a map $C(\Omega\times Z_\omega,\Z)\to
K^0(X^G_{\mathcal{P}(w)})$. Then, if $E$ is a compact-open subset of $Z_\omega$ and
$n$ is an integer, the image of $\chi_{E\times 2^n\Omega}$ under this map
is up to  identification of equation   (\ref{equ-k0-hull}) the class of $\chi_E/2^n$ in $\co C(Z_w,\Z[1/2])$. 
The description of the identification  of equation (\ref{equ-h2-hull}) is
obtained in the same way by using the isomorphism
\begin{equation*}C(\Omega \times Z_\omega,\Z)\cong\check{H}^0(\Omega
  \times Z_\omega,\Z)\stackrel{\cong}{\to}
  \check{H}^2(\Omega \times Z_\omega\times(0,1)^2,\Z)\end{equation*} provided by the cup
product by  the fundamental class of
$\check{H}^2((0,1)^2,\Z)$. Moreover, since the Chern character is
natural and  intertwins Bott periodicity and the cup product by the
fundamental class of $\check{H}_c^1((0,1),\Z)$, we deduce that  up to the
identifications of equations (\ref{equ-k0-hull}) to (\ref{equ-h2-hull}),
it is given by  the identity maps of  $\inv C(Z_w,\Z)\oplus \co
C(Z_w,\Z[1/2])$ and of
$\co C(Z_w,\Z)$.
It is easy to guess how the generators of $K^*(X^G_{\mathcal{P}(w)})$
described in  equations (\ref{equ-k0-hull}) and  (\ref{equ-k1-hull})
should be identified with those of $K_*(C(X^G_{\mathcal{P}(w)})\rt G)$
described in section \ref{section_K_th} 
under Thom-Connes isomorphism of equation (\ref{equ-mt-iso}). Recall first that  for a unital
$C^*$-algebra $A$ provided with an automorphism $\beta$,
\begin{itemize}
\item the mapping torus $A_\beta$ is provided with an action
  $\widehat{\beta}$ of $\R$ by automorphisms
  (see section \ref{subsec-suspension}) and moreover
  $A\rtimes_\beta\Z$ and   $A_\beta\rtimes_{\widehat{\beta}}\R$ are
  Morita  equivalent; 
\item the mapping torus isomorphisms are   the composition of the 
  Thom-Connes isomorphisms $K_0(A_\beta)\stackrel{\cong}{\longrightarrow} K_1(A\rtimes_\beta\Z)$
  and $K_1(A_\beta)\stackrel{\cong}{\longrightarrow} K_0(A\rtimes_\beta\Z)$ with the isomorphism 
$K_*(A_\beta\rtimes_{\widehat{\beta}}\R)\cong
K_*(A\rtimes_{{\beta}}\Z)$ induced with the above Morita-equivalence;
\end{itemize}
It is then straightforward to check that viewing
$C(X^G_{\mathcal{P}(w)})\rt G$ as a double crossed product by $\Z$ as we
did in section \ref{sec-c*-tiling}, the Thom-Connes isomorphism
$$K_*(C(X^G_{\mathcal{P}(w)}))\stackrel{\cong}{\to}
K_*(C(X^G_{\mathcal{P}(w)})\rt G)$$ is obtained by using twice the
mapping torus isomorphim (up to stabilisation for the second one).
In view of our purpose of idenfying the generators of 
$K_*(C(X^G_{\mathcal{P}(w)}))$ with those of $K_*(C(X^G_{\mathcal{P}(w)})\rt G)$,
we will need the following alternative description of the mapping
torus isomorphism using the bivariant Kasparov $K$-theory groups  $KK_*^\Z(\bullet,\bullet)$ \cite{Ka}. Let $A$ be a $C^*$-algebra and let $\beta$ be an
automorphism of $A$. Since the action of $\Z$ on $\R$ by translations is free and
proper, we have a Morita equivalence between $A_\beta$ and
$C_0(\R,A)\rtimes\Z$, where $C_0(\R,A)\cong C_0(\R)\otimes A$ is
equipped with the diagonal action of $\Z$. Recall that the
$\Z$-equivariant  unbounded
operator $\imath\frac{d}{dt}$ of $L^2(\R)$ gives rise to a unbounded
$K$-cycle and hence  to an element $y$ in $KK^\Z_1(C_0(\R),\C)$. Then
the mapping torus isomorphism of equation (\ref{equ-mt-iso}) is the
composition
\begin{equation*}K_*(A_\beta)\stackrel{\cong}{\longrightarrow}
  K_*(C_0(\R,A)\rtimes\Z)\stackrel{\otimes_{C_0(\R,A)\rtimes\Z}J_\Z(\tau_A(y))}{\longrightarrow}K_{*+1}(A\rtimes_\beta\Z)\end{equation*}
where,
\begin{itemize}
\item the first map comes from the Morita equivalence;
\item $J_\Z:KK_*^\Z(\bullet,\bullet)\to
  KK_*(\bullet\rtimes\Z,\bullet\rtimes\Z)$ is the Kasparov
  transformation in bivariant $KK$-theory;
\item for any $C^*$-algebra $B$ equipped with an action of $\Gamma$ by
  automorphism $\tau_B:KK^\Gamma_*(\bullet,\bullet)\to
  KK^\Gamma_*(\bullet\otimes B,\bullet \otimes B)$ is the tensorisation
  operation;
\item $\otimes_{C_0(\R,A)\rtimes\Z}J_\Z(\tau_A(y))$
stands for the right Kasparov product  by
$J_\Z(\tau_A(y))$.
\end{itemize}

 Then the
identification  between  the generators of $K^*(X^G_{\mathcal{P}(w)})$ and of
$K_*(C(X^G_{\mathcal{P}(w)})\rtimes G)$  can be achieved using 
the next two  lemmas.
\begin{lemma}\label{lem-inv}Let $A$ be a unital $C^*$-algebra together with an
  automorphism $\beta$.
 Let $e$ be invariant projector in $A$ and let $x_e$ be the class in
 $K_0(A_\beta)$ of  
 the  projector  $[0,1]\to A;\, t\mapsto e$. Then the image of $x_e$
 under the mapping torus isomorphism is equal to the
  class of the unitary $1-e+e\cdot u$ of
  $A\rtimes_\beta\Z$ in $K_1(A\rtimes_\beta\Z)$  (here $u$ is the unitary
  of $A\rtimes_\beta\Z$ corresponding to the positive generator of $\Z$);
\end{lemma}
\begin{proof}
The invariant projector $e$ gives rise to an equivariant map 
$\C\to A;\, z\to ze$ and hence to a homomorphism $C(\mathbb{T})\to\, A_\beta$. By naturality of the mapping torus, this amounts to prove
the result for $A=\C$ which is done in \cite[Example 6.1.6]{valette}.
\end{proof}

\begin{lemma}Let $A$ be a unital $C^*$-algebra together with an
  automorphism $\beta$
and let $x$ be an element in $K_*(A)$. The two following
  elements then coincide: 
\begin{itemize}
\item the image of $x$ under
  the composition $$K_*(A)\stackrel{\cong}{\to}K_{*+1}(C((0,1),A))\to
  K_{*+1}(A_\beta)\to  K_{*+1}(A\rtimes_\beta\Z),$$ where
\begin{itemize}
\item the first map is the Bott periodicity isomorphism;
\item the second map is induced by the inclusion
  $C((0,1),A)\hookrightarrow A_\beta$;
\item the third map is the mapping torus isomorphism.
\end{itemize}
\item  the image of $x$ under the map 
$K_*(A)\to K_*(A\rtimes_\beta\Z)$ induced by the inclusion
$A\hookrightarrow A\rtimes_\beta\Z$.
\end{itemize}
\end{lemma}
 \begin{proof}
Let us first describe the imprimity bimodule implementing the Morita
equivalence between $A_\beta$ and $C_0(\R,A)\rtimes\Z$. Indeed, in a
more general setting, if
\begin{itemize}
\item $X$ is a locally compact space equipped with a proper  action of $\Z$
  by  homeomorphisms,
\item $B$ is a $C^*$-algebra provided with an action of $\Z$ by
  automorphisms;
\item  $B_X^\Z$ stands for the  algebra of equivariant continuous maps
  $f:X\to B$ such that $\,\Z.x\mapsto \|f(x)\|$ belongs to $C_0(X/\Z)$.
\end{itemize}
then, if we equip  $C_0(X,B)\cong C_0(X)\otimes B$ with the diagonal
action of $\Z$, there is an imprimitivity $B_X^\Z-C_0(X,B)\rtimes\Z$-bimodule
defined in the following way:
 let us consider on $C_c(X,B)$ the
$C_0(X,B)\rtimes\Z$-valued inner product
\begin{equation*}
\langle \xi,\xi'\rangle(n)=n(\xi^*)\xi',
\end{equation*}
for $\xi$ and $\xi'$ in $C_c(X,B)$ and $n$ in $\Z$. This inner product
is namely positive and gives rise to    a right
$C_0(X,B)\rtimes\Z$-Hilbert module $\cE(B,X)$, the action of
$C_0(X,B)\rtimes\Z$ on the right being given for $\xi$ in $C_c(X,B)$
and $h$ in $C_c(\Z\times X,B)\subset C_c(\Z,C_0(X,B))$ by 
\begin{equation*}
\xi\cdot h(t)=\sum_{n\in\Z}n(\cE(n+t))n(h(n,n+t)).
\end{equation*}
The action by pointwise multiplication of $B_X^\Z\subset
C_b(X,B)$  on $C_0(X,B)$ extends to a left  $B_X^\Z$-module structure
on  $\cE(B,X)$. Let us denote by $[\cE(B,X)]$ the class of the
$B_X^\Z-C_0(X,B)\rtimes\Z$-bimodule  $\cE(B,X)$ in
$KK_*(B_X^\Z,C_0(X,B)\rtimes\Z)$. It is straightforward to check
that 
\begin{itemize}
\item $[\cE(B,X)]$ is natural in both variable, in particular, if $Y$
  is an open invariant subset of $X$ and let us denote  by
  $\iota_{Y,X,B}:C_0(Y,B)\to C_0(X,B)$ and
  $\iota^\Z_{Y,X,B}:B_Y^\Z\to B_X^\Z$ the homomorphisms
  induced by the inclusion $Y\hookrightarrow X$ and  respectively by 
 $[\iota_{Y,X,B}]$ and $[\iota^\Z_{Y,X,B}]$ the corresponding classes  
 in $KK^\Z_*(C_0(Y,B),C_0(X,B))$ and $KK_*(B_Y^\Z,B_X^\Z)$,
 then
$$[\iota^\Z_{Y,X,B}]\otimes_{B_X^\Z}[\cE(B,X)]=[\cE(B,Y)]\otimes_{C_0(Y,B)}J_\Z([\iota_{Y,X,B}]).$$
\item up to the identification $B_\Z^\Z\cong B$, the class
  $[\cE(B,\Z)]$ is induced by the composition  $$B\longrightarrow
  C_0(\Z,B)\hookrightarrow  C_0(\Z,B)\rtimes
  \Z$$ where the first map is $b\mapsto \delta_{0}\otimes b$.
\item if $V$ is any locally compact space, and  if we consider $\Z$
  acting trivially on it, then up to the identifications $B_{X\times
    V}^\Z\cong B_{X}^\Z\otimes C_0(V)$ and $C_0(X\times V,B)\rtimes
  \Z\cong C_0(X,B)\rtimes\Z\otimes C_0(V)$,
  we have $[\cE(B,V\times X)]=\tau_{C_0(V)}([\cE(B, X)])$.
\end{itemize}
Noticing that for a $C^*$-algebra $A$ provided with an automorphism
$\beta$, we have a natural identification  $A_\beta\cong
A_\R^\Z$, the mapping torus isomorphism of equation
(\ref{equ-mt-iso})  is obtained by right  Kasporov product with $[\cE(A,
\R)]\otimes_{C_0(\R,A)\rtimes\Z}J_\Z(\tau_A([y]))$.
  From this, we see that the composition in the statement of the
  lemma is given by right Kasparov product with
\begin{equation}\label{kasp_prod}
z=\tau_{A}([\partial])\otimes_{C_0((0,1),A)}[\iota^\Z_{(0,1)\times\Z,\R,A}]\otimes_{A_\beta}[\cE(A,
\R)]\otimes_{C_0(\R,A)\rtimes\Z}J_\Z(\tau_A(y))
\end{equation}
where
\begin{itemize}
\item $[\partial]$ in $KK_1(\C,C_0(0,1))$ is the boundary of the
  evaluation at $0$  extension
  $$0\to C_0(0,1)\to C_0[0,1)\to \C\to 0;$$
\item we have used the identification $(0,1)\times \Z\cong\R\setminus
  \Z$ to see  $(0,1)\times \Z$ as an invariant open subset of $\R$, in
  particular we have $A_{(0,1)\times\Z}^\Z\cong C_0((0,1),A)$.
\end{itemize}
According to the naturality properties of $[\cE(\bullet,A)]$ listed
above, we get that
\begin{eqnarray}\label{equ-kasp-prod}
\nonumber[\iota^\Z_{(0,1)\times\Z,\R,A}]\otimes_{A_\beta}[\cE(\R,A)]&=&[\cE((0,1)\times\Z,A)]\otimes_{C_0(\Z,A)\rtimes\Z\otimes
C_0(0,1)}J_\Z([\iota_{(0,1)\times\Z,\R,A}])\\
&=&\tau_{C_0(0,1)}([\cE(\Z,A)])\otimes_{C_0(\Z,A)\rtimes\Z\otimes
C_0(0,1)}J_\Z([\iota_{(0,1)\times\Z,\R,A}])\end{eqnarray}

Using commutativity of exterior Kasparov product, we get from equation
(\ref{equ-kasp-prod}) that
\begin{equation}
z=[\cE(\Z,A)]\otimes_{C_0(\Z,A)\rtimes\Z}J_\Z(\tau_{C_0(\Z,A)}([\partial])\otimes_{C_0((0,1)\times\Z,A)}[\iota_{(0,1)\times\Z,\R,A}]\otimes
\tau_A(y)).
\end{equation}
Let $y'$ be the element of $KK_1(C_0(0,1),\C)$ corresponding in the
unbounded picture to the unbounded operator $\imath\frac{d}{dt}$ on
$L^2(0,1)$ and let $[\mathcal{F}]$ be the element of $KK^\Z(C_0(\Z),\C)$
corresponding to the equivariant representation by compact operator of $C_0(\Z)$ onto
$\ell^2(\Z)$ (equipped with the left regular representation) given by
pointwise multiplication. Then it is straightforward to check that
$$[\iota_{(0,1)\times\Z,\R,A}]\otimes_{C_0(\R,A)}\tau_A(y)=
\tau_A(\tau_{C_0(\Z)}(y')\otimes_{C_0(\Z)}[\mathcal{F}]).$$But
it is a standard fact that $[\partial]\otimes_{C_0(0,1)}y'=1$ in the ring
$KK_0(\C,\C)\cong\Z$ and hence we eventually  get that
$$z=[\cE(A,\Z)]\otimes_{C_0(\Z,A)\rtimes\Z}J_\Z(\tau_A([\mathcal{F}])).$$
A direct inspection of the right  hand side of this equality shows that
$z$  is indeed the class of $KK_*(A,A\rtimes\Z)$ induced by the
inclusion $A\hookrightarrow A\rtimes\Z.$\end{proof} Recall that in
section  \ref{section_K_th}, we have established isomorphisms 
\begin{equation} \label{equ-id-k0} K_0(C(X^G_{\mathcal{P}(w)})\rtimes G)\cong   \co
C(Z_w,\Z[1/2])\oplus \inv C(Z_w,\Z)\end{equation}
and
\begin{equation} \label{equ-id-k1} K_1(C(X^G_{\mathcal{P}(w)})\rtimes G)\cong
\co C(Z_w,\Z).\end{equation} Under this identification, and 
using lemma
\ref{lem-inv} and twice lemma \ref{lem-coinv}, we are now in position
to describe the image of the generators of $K^*(X^G_{\mathcal{P}(w)})$
under  the double Thom-Connes isomorphism.
\begin{corollary}
Under the identification of equations
(\ref{equ-k0-hull}), (\ref{equ-k1-hull}), (\ref{equ-id-k0}) and (\ref{equ-id-k1}), the double
Thom-Connes isomorphism
$$K^*(X^G_{\mathcal{P}(w)})\stackrel{\cong}{\longrightarrow}K_*(C(X^G_{\mathcal{P}(w)})\rtimes
G)$$ corresponds to the identity maps of  $\co
C(Z_w,\Z[1/2])\oplus \inv C(Z_w,\Z)$ and $\co C(Z_w,\Z)$.
\end{corollary}
\begin{proof}The statement concerning the factor $\inv C(Z_w,\Z)$  is
  indeed a consequence of the discussion at the end of section
  \ref{section_K_th}.
Before proving the statements concerning the factors  $\co
C(Z_w,\Z)$ and $\co
C(Z_w,\Z[1/2])$, we  sum up for convenience of the reader the main features,
described in sections \ref{exemple} and \ref{sec-c*-tiling} of the dynamic of the continuous
hull for the coloured and the  uncoloured Penrose hyperbolic tilings.
\begin{itemize}
\item   the closure $X^\cN_{\mathcal{P}}$ of
  $\cN_{\mathcal{P}}\cdot\mathcal{P}$ in $X^G_{\mathcal{P}}$ for the tiling topolology of the Penrose hyperbolic tiling
  $\mathcal{P}$ is homeomorphic to the suspension of the homeomorphism
  $o:\Omega\to\Omega;\, \omega\mapsto \omega+1$;
\item the continuous hull  $X^G_{\mathcal{P}(w)}$ of the coloured Penrose hyperbolic tiling
  $\mathcal{P}(w)$ is homeomorphic to the suspension of the homeomorphism
  $X^\cN_{\mathcal{P}}\times Z_\omega\to X^\cN_{\mathcal{P}}\times Z_\omega:\,
      (\mathcal{T},w')\mapsto(R\cdot\mathcal{T},\sigma(w'))$;
 \item If we provide  $X^\cN_{\mathcal{P}}\times Z_\omega$ with the
     diagonal action of $\R$, by translations on $X^\cN_{\mathcal{P}}$
       and trivial on  $Z_\omega$, and equip the groupoid
       $\mathcal{G}=(X^\cN_{\mathcal{P}}\times Z_\omega)\rtimes \R$
         with the Haar system arising from the Haar measure of $\R$,
         then $C(X^G_{\mathcal{P}(w)})\rtimes
\R$ is the mapping torus  of $C_r^*(\mathcal{G},\lambda)$ with respect
to the automorphim $\tilde{\alpha_\omega}$   arising from the automorphism of groupoid
$\mathcal{G}\to\mathcal{G};\,(\mathcal{T},w',t)\mapsto(R\cdot\mathcal{T},\sigma(\omega'),2t)$
(see section \ref{sec-c*-tiling});
\item  $C_r^*(\mathcal{G},\lambda)$ is Morita-equivalent to the crossed
  product $C(\Omega\times Z_\omega)\rtimes\Z$ for the action of $\Z$
  on $C(\Omega\times Z_\omega)$ arising from $o\times Id_{Z_\omega}$
  (see section \ref{sec-penrose-uncolor}).
\end{itemize}
Let us consider the following diagram:
\begin{equation}\label{diag-TC}\begin{CD}
@.  K_i(C_r^*(\mathcal{G},\lambda))@>>>K_i(C(X^G_{\mathcal{P}(w)})\rtimes
G)\\
@.  @A=AA  @AATCA\\
K_i(C(\Omega\times Z_w))@>>> K_i(C_r^*(\mathcal{G},\lambda))@>>>K_{i+1}(C(X^G_{\mathcal{P}(w)})\rtimes
\R)\\
@A=AA @AATCA   @AATCA\\
K_i(C(\Omega\times Z_w))@>>>K_{i+1}(C(X^\cN_{\mathcal{P}}\times Z_w))@>>>K_{i}(C(X^G_{\mathcal{P}(w)}))
\end{CD},\end{equation}
where
\begin{itemize}
\item the bottom and the right middle  horizontal arrows  are the maps
  defined for any $C^*$-algebra $A$ provided by an automorphism
  $\beta$ as the composition 
$$K_i(A)\to K_{i+1}(C_0((0,1),A)\to K_{i+1}(A_\beta)$$ of the Bott
peridocity isomorphism with the homomorphism induced in $K$-theory
by  the inclusion $C_0((0,1),A)\hookrightarrow A_\beta$;
\item the left  middle horizontal arrow is up to the Morita
  equivalence between $C_r^*(\mathcal{G},\lambda)$  and 
  $C(\Omega\times Z_\omega)\rtimes\Z$   induced by the inclusion
  $C(\Omega\times Z_\omega)\hookrightarrow C(\Omega\times
  Z_\omega)\rtimes\Z$;
\item the top  horizontal arrow is up to the Morita equivalence
  between
  $C_r^*(\mathcal{G},\lambda)\rtimes_{\tilde{\alpha_\omega}}\Z$ and $C(X^G_{\mathcal{P}(w)})\rtimes
G$ is induced by the inclusion
$C_r^*(\mathcal{G},\lambda)\hookrightarrow C_r^*(\mathcal{G},\lambda)\rtimes_{\tilde{\alpha_\omega}}\Z$.
\item the vertical maps $TC$ stand for the Thom-Connes isomorphisms.
\end{itemize}
Then the inclusion $\co C(Z_w,\Z[1/2])\hookrightarrow
K^0(X^G_{\mathcal{P}(w)})$  of equation
(\ref{equ-k0-hull}) is induced by  the
composition of the bottom arrows and
the inclusion   $\co C(Z_w,\Z[1/2])\hookrightarrow
K_0(C(X^G_{\mathcal{P}(w)})\rtimes G)$ of equation (\ref{equ-id-k0}) 
 is induced by the upper staircase. According to lemma \ref{lem-inv},
 the left bottom and the right top
 squares are commutative. Hence, the proof of the statements regarding
 to the $\co C(Z_w,\Z[1/2])$  summand  amounts
 to show that the bottom right square is commutative.
To see this, let us equip $X^\cN_{\mathcal{P}}\times Z_w\times [0,1)$
with the action of $\R$ by homeomorphisms 
$$X^\cN_{\mathcal{P}}\times Z_w\times [0,1)\times\R\longrightarrow
X^\cN_{\mathcal{P}}\times Z_w\times
[0,1);\,(\mathcal{T},\omega',s,t)\mapsto
(\mathcal{T}+2^{-s}t,\omega',s).$$
If we restrict this action to $X^\cN_{\mathcal{P}}\times Z_w\times
(0,1)$, then the inclusion $X^\cN_{\mathcal{P}}\times Z_w\times
(0,1)\hookrightarrow X^G_{\mathcal{P}(w)}$ is $\R$-equivariant and the
Bott periodicity isomorphism is the boundary of the equivariant short
exact sequence
\begin{equation}\label{equ-bott}
0\to C_0(X^\cN_{\mathcal{P}}\times Z_w\times
(0,1))\to C_0(X^\cN_{\mathcal{P}}\times Z_w\times
[0,1))\to C_0(X^\cN_{\mathcal{P}}\times Z_w)\to 0
\end{equation} provided by evaluation at $0$. This equivariant short exact sequence gives rise to a short
exact sequence for crossed products
\begin{equation}\label{equ-bott-equ}
0\to C_0(X^\cN_{\mathcal{P}}\times Z_w\times
(0,1))\rtimes\R\to C_0(X^\cN_{\mathcal{P}}\times Z_w\times
[0,1))\rtimes\R\to C_0(X^\cN_{\mathcal{P}}\times Z_w)\rtimes\R\to 0.
\end{equation}
and since the Thom-Connes isomorphism is natural, it intertwins the
corresponding boundary maps and hence we get a commutative diagram

\begin{equation}\label{diag-bott}
\begin{CD}
 K_{i+1}(C(X^\cN_{\mathcal{P}}\times Z_w)\rtimes\R)@>>>K_i(C_0(X^\cN_{\mathcal{P}}\times Z_w\times
(0,1)\rtimes\R)@>>>K_{i}(C(X^G_{\mathcal{P}(w)})\rtimes\R)
\\
  @ATCAA  @AATCA  @AATCA\\
K_i(C(X^\cN_{\mathcal{P}}\times Z_w))@>>>K_{i+1}(C_0(X^\cN_{\mathcal{P}}\times Z_w\times
(0,1))@>>>K_{i+1}(C(X^G_{\mathcal{P}(w)}))
\end{CD},\end{equation}where the left horizontal arrows are induced by the
boundary maps corresponding to the exact sequences of equation
(\ref{equ-bott}) and (\ref{equ-bott-equ}) and the right horizontal arrows
are induced by the equivariant inclusion $X^\cN_{\mathcal{P}}\times Z_w\times
(0,1)\hookrightarrow X^G_{\mathcal{P}(w)}$.
Let us consider the family  groupoids $(0,1)\times\G$ and
$[0,1)\times\G$. Notice that if  $X^G_{\mathcal{P}(w)}\rtimes\R$ is
viewed as the suspension of the groupoid $\G$, then  $(0,1)\times\G$  is the
restriction of $X^G_{\mathcal{P}(w)}\rtimes\R$ to a fundamental
domain. The reduced $C^*$-algebras of these two groupoids are  respectively
$C_0((0,1),C_r^*(\mathcal{G},\lambda)))$ and
$C_0([0,1),C_r^*(\mathcal{G},\lambda))$ and the
automorphism of groupoids $$[0,1)\times\G\to (X^\cN_{\mathcal{P}}\times Z_w\times
[0,1))\rtimes\R;\, (\mathcal{T},\omega',s,t)\mapsto
(\mathcal{T},\omega',s,2^st)$$
gives rise to a commuting diagram
{\tiny{$$\begin{CD}
0@>>>C_0((0,1),C_r^*(\mathcal{G},\lambda))@>>>C_0([0,1),C_r^*(\mathcal{G},\lambda))@>>>C_r^*(\mathcal{G},\lambda)@>>> 0\\
@. @VVV @VVV @VV=V\\
0@>>>C_0(X^\cN_{\mathcal{P}}\times Z_w\times
(0,1))\rtimes\R@>>>C_0(X^\cN_{\mathcal{P}}\times Z_w\times
[0,1))\rtimes\R@>>>C_0(X^\cN_{\mathcal{P}}\times Z_w
)\rtimes\R@>>>0. 
\end{CD}$$}}Using naturality of the boundary map, we see that the
composition of the top  horizontal arrows in diagram (\ref{diag-bott}) is
the composition $$K_i(C_r^*(\mathcal{G},\lambda))\to K_{i+1}(C_0((0,1),C_r^*(\mathcal{G},\lambda))\to K_{i+1}(C(X^G_{\mathcal{P}(w)})\rtimes\R))$$ of the Bott
peridocity isomorphism with the homomorphism induced in $K$-theory by  the inclusion
$C_0((0,1),C_r^*(\mathcal{G},\lambda))\hookrightarrow C(X^G_{\mathcal{P}(w)})\rtimes\R$. This concludes the proof for the
statement concerning the summand $\co C(Z_\omega,\Z[1/2])$. The
statement concerning the summand  $\co C(Z_\omega,\Z)$ is a
consequence of the commutativity of the top square of diagram
\ref{diag-TC} and of lemma \ref{lem-inv} applied to the middle bottom vertical
arrow.
\end{proof}

\section{The  cyclic cocycle associated to a harmonic probabilty}
Recall that according to the discussion ending section
\ref{subsec-ergotic}, a
 probability is  harmonic if and only if it is   $G$-invariant. In this section, we associate to a harmonic probability
a $3$-cyclic cocycle on the smooth crossed product algebra of $\xg\rt
G$. This  cyclic cocycle is indeed  builded from a $1$-cyclic cocycle
on the algebra of smooth (along the leaves)  functions on $\xg$ by
using the analogue in cyclic cohomology  of the Thom-Connes
isomorphism (see \cite{enn}). We
give a description of this cocycle and we discuss an odd version of
the gap-labelling.
\subsection{Review on smooth crossed products}
We collect here results from \cite{enn} concerning smooth crossed products
by an action of $\R$ that we will need later on. 

\smallskip
Let $\af$ be a
Frechet algebra with respect to an increasing  family of semi-norms
$(\|\bullet\|_k)_{k\in\N}$. 
\begin{definition}
A smooth action on $\af$ is a homomorphism $\alpha:\R\to
\operatorname{Aut}\,\af$ such that 
\begin{enumerate}
\item For every $t$ in $\R$ and $a$ in $\af$, the function $t\mapsto\alpha_t(a)$ is
  smooth.
\item For every integers $k$ and $m$, there exist  integers $j$ and $n$
  and  a
  real $C$ such that $\left\|\frac{d^k}{dt^k}\alpha_t(a)\right\|_m\leq
  C(1+t^2)^{j/2}\|a\|_n$ for every $a$ in $\af$. 
\end{enumerate}
\end{definition}
If $\alpha$ is a smooth action on $\af$, then the smooth crossed product
 $\af\rt_\al\R$ is defined as the set of smooth functions 
$f:\R\to\af$   such that   
$$\|f\|_{k,m,n}\stackrel{\text{def}}{=\!=}\sup_{t\in\R}(1+t^2)^{k/2}\left\|\frac{d^m}{dt^m}f(t)\right\|_n<+\infty$$
  for all integers $k,m$ and $n$. The smooth crossed product
  $\af\rt_\al\R$ provided with the family of semi-norm
  $\|\bullet\|_{k,m,n}$ for $k,m$ and $n$ integers together with  the convolution product
$$f*g(t)=\int f(s)\al_s(g(t-s))dt$$ is then a Frechet algebra.
Notice that a smooth action $\al$ on a Frechet algebra $\af$ gives
rise to a bounded derivation $Z_\al$ of $\af\rt_\al\R$ defined by
$Z_\al(f)(t)=tf(t)$ for all $f$ in $\af\rta\R$ and $t$ in $\R$.

\smallskip
Let $\afx$ be the algebra of continuous and smooth along the leaves
functions on $\xg$, i.e  functions whose restrictions to leaves admit
 at all order differential which are continuous as functions on $\xg$.
Let $\be^0$ and $\be^1$ be the two actions of $\R$ on $\afx$
respectively induced by 
$$\R\times\xg\to\xg;\,(t,\T)\mapsto \T+t$$ and
$$\R\times\xg\to\xg;\,(t,\T)\mapsto  R_{2^{t}}\cdot \T .$$ Let $X$ and $Y$
be respectively the vector fields associated to $\be^0$ and $\be^1$.
Then $\afx$ is a Frechet algebra with respect to the family of
semi-norms
$$f\mapsto \sup_\xg|X^kY^l(f)|,$$ where  $k$ and $l$
run through
integers. It is clear that  $\be^0$ is a smooth action on $\afx$. Moreover, 
$$\R\times\afx\rtimes_{\beta^0}\R\to\afx\rtimes_{\beta^0}\R;(t,f)\mapsto [s\mapsto \be^1(f(2^{-t}s))]$$ is an
action of $\R$ on $\afx\rt_{\be^0}\R$ by automorphisms. This action is
not smooth in the previous sense. Nevertheless, the action $\be^1$
satisfies conditions (1),(2) and (3) of \cite[Section 7.2]{enn} with respect
to the family of functions $\rho_n:\R\to\R;t\mapsto 2^{2n|t|}$, where
$n$ runs through integers. In this situation, we can define the smooth
crossed product  $\afx\rt_{\be^0}\R\rt^\rho_{\be^1}\R$ of $\afx\rt_{\be^0}\R$ by $\be^1$ to be the
  set of smooth functions\\
$f:\R\to\afx\rt_{\be^0}\R$   such that   
$$\|f\|_{k,l,m}\stackrel{\text{def}}{=\!=}
\sup_{t\in\R}\rho_k(t)\left\|\frac{d^l}{dt^l}f(t)\right\|_{m}<+\infty$$
  for all integers $k,l$ and $m$ (we have reindexed   for convenience the family of
  semi-norms on $\afx\rt_{\be^0}\R$ using integers). Then   $\afx\rt_{\be^0}\R\rt^\rho_{\be^1}\R$
   provided with the family of semi-norms
  $\|\bullet\|_{k,l,m}$ for $k,l$ and $m$ integers together with  the convolution product
 is  a Frechet algebra. Moreover, this algebra can be viewed as a
 dense subalgebra of $C(\xg)\rtimes G$. As for smooth actions, the action $\be^1$
 gives rise to a derivation $Z_{\be^1}$ of
 $\afx\rt_{\be^0}\R\rt^\rho_{\be^1}\R$ (defined by the same formula).
\subsection{The $3$-cyclic cocycle}
Let $\eta$ be a $G$-invariant probability on $\xg$. Define
$$\tau_{w,\eta}:\afx\times\afx\to\C;\,(f,g)\mapsto\int Y(f)gd\eta.$$Using the Leibnitz
rules and the invariance of $G$, it is straightforward to check that
$\tau_{w,\eta}$ is $1$-cyclic cocycle. 
In \cite{enn} was constructed for a smooth action $\al$ on a Frechet
algebra $\af$ a homomorphism $H_\lambda^n(\af)\to
H_\lambda^{n+1}(\af\rtimes_\al\R)$, where  $H_\lambda^*(\bullet)$ stands
for the cyclic cohomology. This homomorphism 
is indeed induced by a homomorphism at the level of cyclic cocycles
$\sh_\al:Z_\lambda^n(\af)\to Z_\lambda^{n+1}(\af\rt_\al\R)$ and commutes
with the periodisation operator $S$. Hence it gives rise to a
homomorphism in periodic cohomology $HP^*(\af)\to HP^{*+1}(\af\rta\R)$
which turns out to be an isomorphism. This isomorphism is for  periodic
cohomology the analogue of the Thom-Connes isomorphism in $K$-theory.

\smallskip

We give now the description of $\sh_{\be^0}\tau_{w,\eta}$.
Let us define first 
$$\xb:\afx\rt_{\be^0}\R\to\afx\rt_{\be^0}\R$$ and $$\yb:\afx\rt_{\be^0}\R\to\afx\rt_{\be^0}\R$$
respectively by 
$\xb f(t)=X(f)(t)$ and $\yb f(t)=Y(f)(t)$, for all $f$ in
$\afx\rt_{\be^0}\R$ and $t$ in $\R$.
Using the relation $Y\circ\be^0_t=\be^0_t\circ Y-t\ln 2\,\be^0_t\circ X$
and applying the definition of $\sh_\be^0$ (see \cite[section 3.3]{enn}), we get:
\begin{proposition}
For any elements  $f$, $g$ and $h$ in $\afx\rt_{\be^0}\rtimes\R$, we have
\begin{enumerate}
\item \begin{align*}
\sh_{\be^0}\tau_{w,\eta}&(f,g,h)=-2\pi i\eta(\yb f*g*\zb
  h(0)+\zb f*g*\yb h(0)\\-2\pi i\ln 2&(\eta(1/2\zb^2 f*g*\xb
  h(0)+\zb f*\zb g * \xb(h)(0)-1/2\xb f*g*\zb^2(0))
\end{align*}
\item
  $$\sh_{\be^0}\tau_{w,\eta}(\be^1_tf,\be^1_tg,\be^1_th)=\sh_{\be^0}\tau_{w,\eta}(f,g,h)$$
  for all $t$ in $\R$, i.e the cocycle $\sh_{\be^0}\tau_{w,\eta}$  is $\be^1$-invariant.
\end{enumerate}
\end{proposition}
According to \cite[Section 7.2]{enn}, the action $\be^1$ on
$\afx\rt_{\be^0}\rtimes\R$
also gives rise to a 
homomorphism
$\sh_{\be^1}:Z_\lambda^n(\afx\rt_{\be^0}\R)\to
Z_\lambda^{n+1}(\afx\rt_{\be^0}\R\rt^\rho_{\be^1}\R)$ which induces an
isomorphim  $HP^*(\afx\rt_{\be^0}\R)\stackrel{\cong}{\longrightarrow}
HP^{*+1}(\afx\rt_{\be^0}\R\rt^\rho_{\be^1}\R)$.
A direct application of the definition of $\sh_{\beta^1}$ leads to
\begin{lemma}
Let $\phi$ be a $\beta_1$-invariant $3$-cyclic cocycle for  $\afx\rt_{\be^0}\R$. Let us
define for any $f,g$ and $h$ in  $\afx\rt_{\be^0}\R\rt^\rho_{\be^1}\R$.
$$\widetilde{\phi}(f,g,h)=2\pi\imath\int_{t_0+t_1+t_2=0}f(t_0)\be^1_{t_0}g(t_1)\be^1_{-t_2}(t_2).$$
Then 
\begin{align*}\sh_{\be^1} \phi(f_0,f_1,f_2,f_3)=-\widetilde{\phi}(f_0,f_1&,f_2*\zbu f_3)+\widetilde{\phi}(\zbu
f_0*f_1,f_2,f_3)\\&-\widetilde{\phi}(f_0,\zbu f_1*f_2,
f_3)-\widetilde{\phi}(\zbu f_0,f_1*f_2,f_3)
\end{align*}
\end{lemma}
\begin{definition}
With above notations, the $3$-cyclic cocycle on
$\afx\rt_{\be^0}\R\rt^\rho_{\be^1}\R$ associated to the Penrose
hyperbolic tiling coloured by $w$ and to a $G$-invariant
probability $\eta$ on $\xg$ is
$$\phi_{w,\eta}=\sh_{\be^1}\sh_{\be^0}\tau_{w,\eta}.$$
\end{definition}
Notice that if we carry out  this  construction for  a  tiling $\T$ of the Euclidian space with
continuous hull ${{X}}^{\R^2}_\T$ with respect to the $\R^2$-action by
translations, we get taking twice the crossed product by $\R$  a  
 $3$-cyclic cocycle which  is indeed equivalent (via the periodisation
 operator) to the  $1$-cycle cocycle
on $C({{X}}^{\R^2}_\T)\rtimes\R^2\cong (C({{X}}^{\R^2}_\T)\rtimes\R)\rtimes\R$ arising from
the trace on $C({{X}}^{\R^2}_\T)\rtimes\R$ associated to an $\R$-invariant
probability on ${{X}}^{\R^2}_\T$.

\smallskip

 The class $[\phi_{w,\eta}]$ of $\phi_{w,\eta}$ in
$HP^1(\afx\rt_{\be^0}\R\rt^\rho_{\be^1}\R)$ is the image of the class of
$\tau_{w,\eta}$ under the composition of isomorphism
$$HP^1(\afx)\stackrel{\cong}{\longrightarrow}HP^0(\afx\rt_{\be^0}\R)\stackrel{\cong}{\longrightarrow}HP^1(\afx\rt_{\be^0}\R\rt^\rho_{\be^1}\R).$$
Since pairing with periodic cohomology provides linear forms for
$K$-theory groups, the $3$-cyclic cocycle $\phi_{w,\eta}$ provides a linear map
$$\phi_{w,\eta,*}:K^1(\afx\rt_{\be^0}\R\rt^\rho_{\be^1}\R)\to\C;\, x\mapsto \langle
[\phi_{w,\eta}],x \rangle.$$  The main issue in computing
$\phi_{w,\eta,*}(K^1(\afx\rt_{\be^0}\R\rt^\rho_{\be^1}\R))$ is that the Thom-Connes
isomorphism a priori may  not hold for $K^1(\afx\rt_{\be^0}\R)$. If it
{were } the case,  the inclusion 
$\afx\rt_{\be^0}\R\rt^\rho_{\be^1}\R\hookrightarrow C(\xg)\rtimes G$
would induces an isomorphism
$K_1(\afx\rt_{\be^0}\R\rt^\rho_{\be^1}\R)\stackrel{\cong}{\longrightarrow}
K_1(C(\xg)\rtimes G)$ and from this we could get that 
$$\phi_{w,\eta,*}(K_1(\afx\rt_{\be^0}\R\rt^\rho_{\be^1}\R))=\Z[\hat{\eta}]\stackrel{\text{def}}{=\!=}\{\hat{\eta}(E),\,E\text{
    compact-open subset of  }Z_w\},$$ where $\hat{\eta}$ is  the probability on 
$Z_w$ of proposition
\ref{carmesureinv} in one-to-one correspondance with $\eta$.

Since  $\Z[\hat{\eta}]$   is indeed the one dimension gap-labelling
for the subshift  corresponding to $w$, this would  be viewed  as an odd version of the gap labelling. Nevertheless, the right
setting to state this generalisation of the  gap-labelling seems to
be the Frechet algebra and a natural question is whether
  we have 
 $$\{\langle
  [\phi_{w,\eta}],x\rangle;\,x\in
  K_1(\afx\rt_{\be^0}\R\rtimes_{\be^1}^\rho\R) \}=\Z[\hat{\eta}]$$ or
  if the pairing bring in new invariants.

\section*{Aknowledgments}This paper was partially written during a 
visit of the first author at the   Pacific
 Institut for Mathematical Sciences in  Victoria, Canada. He is also 
 indebted  to J. Renault for some very helpful
discussions on  Haar systems for groupoids.


\begin{thebibliography}{AA}
\bibitem{AP}{\sc J.E. Anderson, I.F. Putnam}. {\it Topological invariants for substitution tilings and their associated $C^*$-algebras}, Ergod. Th. \& Dyn. Syst. {\bf 18} (1998), 509-537.

\bibitem{Bel}{\sc J. Bellissard}. {\it  Gap labelling theorems for Schr\"odinger operators, From number theory to physics} (Les Houches, 1989), 538?630, Springer, Berlin, 1992.

\bibitem{BBG}{\sc J. Bellissard, R. Benedetti, J.-M. Gambaudo}.
{\it Spaces of tilings, finite telescopic approximations and gap-labelling,}  Comm. Math. Phys.  {\bf 261}  (2006),  no. 1, 1--41. 

\bibitem{BeO}{\sc M. T. Benameur, H. Oyono-Oyono}. { \it  Index theory for quasi-crystals I. Computation of the gap-label group}, Journal of Functional Analysis 252 (2007), 137-170.

\bibitem{G}{\sc R. Benedetti, J. M. Gambaudo}. {\it On the dynamics of ${\mathbb G}$-solenoids. Applications to Delone sets}, Ergod. Th. \& Dyn. Syst. {\bf 23} (2003), 673-691.

\bibitem{enn}{\sc G. A. Elliott, T.  Natsume, R Nest}.
 {\it  Cyclic cohomology for one-parameter smooth crossed products}
 Acta Math. {\bf 160}, (1988), n\textsuperscript o 3-4, 285-305

\bibitem{Ga}{\sc L. Garnett}. {\it Foliations, The Ergodic theorem and brownian motion}, Journ. of Funct. Analysis. {\bf 51} (1983),
285-311

\bibitem{GH} {\sc W.H. Gottschalk, G.A. Hedlund}. {\it Topological dynamics}.
American Mathematical Society Colloquium Publications, Vol. 36. American Mathematical Society, Providence, R. I., 1955

\bibitem{Gh}{\sc \'E. Ghys}. {\it lamination par surface de Riemann}, Dynamique et g\'eom\'etrie complexes, Panoramas \& Synth\`ese {\bf 8} (1999), 49-95

\bibitem{KaPu}{\sc J. Kaminker, I. Putnam}. {\it A proof of the gap labelling conjecture}, Michigan Mathematical Journal {\bf 51}, (2003), no 3, 537--546. 

\bibitem{Ka}{\sc G. G., Kasparov}. {\it Equivariant $KK$-theory and the Novikov conjecture.} Invent. Math. {\bf 91} (1988), no. 1, 147--201.

\bibitem{KP}{\sc J. Kellendonk, I.F. Putnam}. {\it Tilings, $C^*$-algebras and $K$-theory}, Directions in Mathematical Quasicrystals, CRM Monograph Series {\bf 13} (2000),
177-206, M.P. Baake \& R.V. Moody Eds., AMS Providence.
\bibitem{MarMoz}{\sc G. Margulis, S. Mozes}. {\it Aperiodic tiling of the hyperbolic plane by convex polygons}, Israel Journ. of Math. {\bf 107} (1998),
319-325

\bibitem{Pe}{\sc S. Petite}. {\it On invariant measures of finite affine type tilings}, Ergod. Th. \& Dyn. Syst. (2006) {\bf 26}, 1159-1176 
\bibitem{Pen}{\sc R. Penrose}. {\it Pentaplexity}, Mathematical Intelligencer {\bf 2} (1979), 32-37
\bibitem{Pl}{\sc J. F. Plante} {\it Foliations with measure preserving holonomy}, Ann. of Math. (2){\bf 102}
(1975),  327--361

\bibitem{Pu}{\sc I. Putnam}. {\it The C*-algebras associated with minimal homeomorphisms of the Cantor set}, Pacific J. Math. {\bf 136} (1989), 329-352.

\bibitem{pv}{\sc M. Pimsner, D.  Voiculescu}, {\it  Exact sequences
    for $K$-groups and Ext-groups of certain cross-product $C\sp{\ast}
    $-algebras},  J. Operator Theory 4 (1980), n\textsuperscript o. 1, p. 93--118. 
\bibitem{Qu}{\sc M. Queff\'elec}. {\it Substitution dynamical systems
    --- spectral analysis}, LNM, 1294. Springer-verlag, Berlin, 1987
\bibitem{re}{\sc J. Renault}, {\it A groupoid approach to $C\sp{\ast}$-algebras},
{Lecture Notes in Mathematics}, {793}, {Springer},  {Berlin}.
\bibitem{Ro}{\sc E. A. Robinson, Jr}. {\it Symbolic dynamics and tilings of $\R^d$}, Symbolic dynamics and its applications, Proc. Sympos. Appl. Math.,  {\bf  60}, Amer. Math. Soc., Providence , RI,  2004, p. 81--119. 
 
 
  
\bibitem{Su}{\sc D. Sullivan}. {\it Cycles for the dynamical study of foliated manifolds and complex manifolds}, Invent. Math. {\bf 36} (1976),
225-255.
\bibitem{valette}{\sc A. Valette}. {\it Introduction to the
    {B}aum-{C}onnes conjecture}, {Lectures in Mathematics ETH
    Z\"urich}, {From notes taken by Indira Chatterji}, with an appendix by Guido Mislin, {Birkh\"auser Verlag},
    {Basel},
       {2002}.
\end{thebibliography}
\end{document}